\documentclass[preprint]{imsart}

\RequirePackage[OT1]{fontenc}
\RequirePackage{amsthm,amsmath,amssymb}
\RequirePackage[numbers]{natbib}
\RequirePackage{natbib}
\bibpunct[:]{(}{)}{,}{a}{}{,}
\newcommand{\mc}{\mathcal}

\newcommand{\mbb}{\mathbb}
\newcommand{\mr}{\mathrm}

\newcommand{\argmin}{\mathop{\rm argmin}\limits}

\startlocaldefs
\numberwithin{equation}{section}
\theoremstyle{plain}
\newtheorem{theorem}{Theorem}[section]
\newtheorem{proposition}{Proposition}[section]
\newtheorem{remark}{Remark}[section]
\newtheorem{lemma}{Lemma}[section]
\newtheorem{corollary}{Corollary}[section]
\newtheorem{definition}{Definition}[section]
\endlocaldefs

\def\betamse{E}

\def\cbeta{\hat{\beta}}
\def\ctheta{\hat{\theta}}

\begin{document}
	
	\begin{frontmatter}
		\title{Robust estimation with Lasso when outputs are adversarially contaminated}
		\runtitle{Robust estimation with Lasso}
		
		\begin{aug}

			\author{\fnms{Takeyuki} \snm{Sasai}\ead[label=e1]{sasai@ism.ac.jp}}
			\address{Department of Statistical Science, The Graduate University for Advanced Studies, SOKENDAI, Tokyo, Japan. 
				\printead{e1}}
			
			\author{\fnms{Hironori} \snm{Fujisawa}
				\ead[label=e3]{fujisawa@ism.ac.jp}
				\ead[label=u1,url]{www.foo.com}
			}
			
			\address{The Institute of Statistical Mathematics, Tokyo, Japan. \\
Department of Statistical Science, The Graduate University for Advanced Studies, SOKENDAI, Tokyo, Japan. \\
Center for Advanced Integrated Intelligence Research, RIKEN, Tokyo, Japan. \\
				\printead{e3}
			}
			
			\thankstext{t1}{Some comment}
			\thankstext{t2}{First supporter of the project}
			\thankstext{t3}{Second supporter of the project}
			\runauthor{Sasai and Fujisawa}
			
		\end{aug}

		\begin{abstract}
			We consider robust estimation when outputs are adversarially contaminated. \cite{NguTra2012Robust} proposed an extended Lasso for robust parameter estimation and then they showed the convergence rate of the estimation error. Recently, \cite{DalTho2019Outlier} gave some useful inequalities and then they showed a faster convergence rate than \cite{NguTra2012Robust}. They focused on the fact that the minimization problem of the extended Lasso can become that of the penalized Huber loss function with $L_1$ penalty. The distinguishing point is that the Huber loss function includes an extra tuning parameter, which is different from the conventional method. We give the  proof, which is  different from \cite{DalTho2019Outlier} and then we give the same convergence rate as \cite{DalTho2019Outlier}.  The significance of our proof is to use some specific properties of the Huber function. Such techniques have not been used in the past proofs. 
		\end{abstract}
		
		
		\tableofcontents
	\end{frontmatter}
	
\section{Introduction}
	We consider a  linear regression model in a high dimensional case, given by
	\begin{align*}
	y_i = X_i^\top\beta^*+\xi_i, \quad i=1,\cdots, n,
	\end{align*}
	where $(X_1,y_1), \cdots, (X_n,y_n) \in \mbb{R}^{d} \times \mbb{R}$ are input-output pairs, $\beta^*$ is the true regression coefficient vector, and $\xi_1,\cdots, \xi_n \in \mbb{R}$ are noises.  The most popular sparse modeling is the least absolute shrinkage and selection operator (Lasso, \cite{Tib1996Regression}) by virtue of its generality and convexity. 
	Lasso estimates the true regression coefficient $\beta^*$ by solving the following convex optimization problem with the tuning parameter $\lambda_s$:
	\begin{align}
	\label{o:normal-lasso}
	\tilde{\beta} \in \argmin_{\beta\in \mbb{R}^d} \left\{\frac{1}{2n}\|Y-X\beta\|_2^2+\lambda_s \|\beta\|_1\right\},
	\end{align}
	where $Y = (y_1,\cdots,y_n)^\top$ and $X=\left(X_1,\cdots ,X_n\right)^\top$.
	Many parameter estimation methods with sparsity on $\beta^*$ have been proposed by \cite{Tib1996Regression,Fan2001Variable,ZouHas2005Regularization,YuaLin2006Model,CanTao2007Dantzig,BelLecTsy2018Slope}, and so on.

Suppose that the outputs may be contaminated by an adversarial noise $\sqrt{n}\theta^* \in \mbb{R}^n$, where 

non-zero entries of $\theta^*$ can take arbitrary values and $\sqrt{n}$ is used for normalization. In this case,  $Y$ is replaced by $Y + \sqrt{n}\theta^*$. Then, the optimization problem (\ref{o:normal-lasso}) may not give an appropriate estimate of $\beta^*$ due to adversarial contamination. 
	To weaken an adverse effect of adversarial contamination, \cite{NguTra2012Robust} proposed an extended Lasso, which estimates $\beta^*$ and $\theta^*$ simultaneously by solving the following convex optimization problem with two tuning parameters $\lambda_s$ and $\lambda_o$:
	\begin{align}
	\label{o:robust-lasso}
	(\cbeta,\ctheta) \in \argmin_{(\beta,\theta) \in \mbb{R}^d \times \mbb{R}^n } \left\{\frac{1}{2n}\|Y-X\beta-\sqrt{n}\theta\|_2^2+\lambda_s \|\beta\|_1+\lambda_o \|\theta\|_1\right\}.
	\end{align}

	Suppose that $X_1, \cdots, X_n$ and $\xi_1,\cdots,\xi_n$ are i.i.d. random samples from $\mc{N}(0,\Sigma)$ and $\mc{N}(0,\sigma^2)$, respectively, and $X_i$s and $\xi_i$s are mutually independent. In this  paper, we assume $d\geq 3$. Let $\|a\|_0$ be the number of non-zero components of $a$. Assume that $\|\beta^*\|_0 \le s$ and $\|\theta^*\|_0 \le o$. 
	\cite{NguTra2012Robust} derived the convergence rate of $\|\beta^*-\cbeta\|_2 + \|\theta^*-\ctheta\|_2$, which implies the convergence rate of $\|\beta^*-\cbeta\|_2$, given by
	\begin{align}
	O\left(\sqrt{\frac{s \log d}{n}} + \sqrt{\frac{o}{n} \log n}\right). \label{e:O_NguyenTran}
	\end{align}
	Recently, \cite{DalTho2019Outlier} insisted a faster convergence rate using different tuning parameters from \cite{NguTra2012Robust}, given by
	\begin{align}
	O\left(\sqrt{\frac{s \log d}{n}} + \frac{o}{n}\sqrt{\log n \log \frac{n}{o}}\right). \label{e:O_DalalyanThompson}
	\end{align} 

 In this paper, we give a correct proof of the convergence rate with a different technique from \cite{DalTho2019Outlier}. On the other hand, Propositions~3~and~4 of \cite{DalTho2019Outlier} are very attractive, which also play important roles in the proofs of this paper. In the past proofs, the convexity and Lipschitz continuity of the Huber loss function were used, in other words, general properties of the loss function were used. In the proof of this paper, we use a specific property of the Huber loss function. (It changes the behavior at the threshold from a quadratic function to a linear function.) By such a careful analysis, we can give a sharper convergence rate than \cite{NguTra2012Robust}, even when the number of outliers is large.
	
	The convergence rate of robust estimation has been examined rapidly in recent years. 
	First, the robust estimation of the mean and scatter matrix was examined under the Huber's contamination. 
	\cite{CheGaoRen2018robust} derived the minimax rate and proposed a method that achieves the minimax rate. However, the  computation is of exponential time. 
	\cite{LaiRaoVem2016Agnostic} proposed another method that is of polynomial time and achieves the optimal rate up to logarithmic factor. 
	\cite{LaiRaoVem2016Agnostic} has been followed by \cite{DiaKamKanLiMotSte2017Being,DiaKamKanLiMoiSte2018Robustly, DiaKamKanLiMoiSte2019Robust,DiaKanKarPriSte2019Outlier},\cite{CheFiaGe2019High}, and so on. 
	The robust and sparse estimation in linear regression has also been studied under the Huber's contamination. \cite{Gao2020Robust} derived the minimax rate of the regression coefficient estimation, given by 
	\begin{align}
	O\left(\sqrt{\frac{s\log (d/s)}{n}}+ \frac{o}{n}\right). \label{e:O_HuberContamination}
	\end{align}
	The adversarial contamination over the Huber's contamination has also been discussed. 
	\cite{NguTra2012Robust} considered the case where the outputs were adversarially contaminated. The convergence rate is given by \eqref{e:O_NguyenTran}. It is slower than \eqref{e:O_HuberContamination}. This is because the adversarial contamination includes various types of contamination over the Huber's contamination. 
	 \cite{CheCarMan2013Robust} treated the extended case where both outputs and inputs were adversarially contaminated, and then they derived the convergence rate, but it is not optimal and it depends on the true value $\beta^*$.
	The case with a fixed $d$ was also discussed. In this case, we can show a faster convergence rate than that by \eqref{e:O_NguyenTran} and \eqref{e:O_DalalyanThompson}, because $d$ is fixed. \cite{DiaKonSte2019Efficient} considered the case that $\beta^*$ was not sparse  and proposed a new method based on filtering, and showed that the convergence rate is $O(\frac{o}{n} \log (n/o))$.
	\cite{LiuSheLiCar2018High} proposed another new method based on iterative hard thresholding, and showed that the convergence rate is $O(\sqrt{\frac{o}{n} })$ in which the order $\log n$ is omitted. 
	For more information on recent developments in robust estimation, see the survey paper by \cite{DiaKan2019Recent}.
	
	This paper is organized as follows. 
	In Section~\ref{s:outline}, we roughly give the main theorem and the outline of the proof. 
	In Section~\ref{s:preliminary}, we prepare some key propositions to prove the main theorem, including Propositions~3~and~4 of \cite{DalTho2019Outlier}, and arrange the necessary conditions. 
	In Section~\ref{sec:opt-re}, we give a simple relation between $\|\beta^*-\hat{\beta}\|_1$ and $\|\Sigma^{\frac{1}{2}}(\beta^*-\hat{\beta})\|_2$, which plays an important role to obtain the convergence rate of the estimation error of the Lasso. 
	In Section~\ref{sec:huber-overflow}, we investigate a behavior of $C_{cut}$, which has a close relation to a specific property of the Huber loss function. This is a distinguishing point of this paper. 
	In Section~\ref{sec:dir}, we give the outline of the proofs of the key propositions. 
	In Section~\ref{sec:put-together}, we give a rigorous statement of the main theorem and prove the main theorem, combining the results prepared in the previous sections.

\section{Main theorem and outline of proof}
\label{s:outline}

\subsection{Main theorem}

The main theorem is roughly given in the following. A rigorous statement of the main theorem, including detailed conditions, is given in Section~\ref{sec:put-together}. The main theorem is compared with the past theorem from the point of view of the conditions and convergence rate. 

Here we prepare some notations related to $\Sigma$. Let $\rho^2 = \max_i(\Sigma_{ii})$. Let $\lambda_{\min}$ and $\lambda_{\max}$ be the smallest and largest eigenvalues of $\Sigma$, respectively. 

\begin{theorem}
	\label{t:main-pre}
	Suppose that $\Sigma$ satisfies the restricted eigenvalue condition $\mr{RE}(s,5,\kappa)$ (cf. Definition~\ref{def:RE}). 
	Assume that $\delta$ is sufficiently small and $n$ is sufficiently large. Let
	\begin{align}
	\label{e:par-pre}
	\lambda_o = C_{\lambda_o}\sqrt{\frac{2\sigma^2\log (n/\delta)}{n}},\quad
	\lambda_s =  \frac{4\sqrt{2}}{\sqrt{3}} C_{\lambda_s} \lambda_o, 
	\end{align}
	where $C_{\lambda_o}$ is an appropriate numerical constant, 
	\begin{align*}
	C_{\lambda_s} &=C_z+\sqrt{2\frac{o}{s}}g(o), \qquad C_z=\sqrt{3 \frac{\rho^2\sigma^2}{\lambda_o^2 n}\log{ \frac{d}{\delta}}}, \\
 & g(o)=\sqrt{\frac{2}{n}} \left( 4.8+\sqrt{\log \frac{81}{\delta}}\right) 
+ 1.2c_\kappa\sqrt{\frac{ 2\rho^2 s \log d}{n}} +4.8\sqrt{e}\sqrt{\frac{o}{n}}\sqrt{4+\log\frac{n}{o}}.
	\end{align*}
	Assume that $ \lambda_o$ and $\lambda_s$ satisfy
	\begin{align}
	\label{cond0-pre}
	& 8 \max \left(3.6\sqrt{\frac{2\rho^2 \log d}{n}}, 2.4\frac{\lambda_s}{\lambda_o}\sqrt{\frac{2\log n}{n}}\right) \sqrt{  \frac{s}{\kappa^2} + \frac{6.25 o \lambda_o^2}{\lambda_s^2}}  \leq C_{n,\delta},
	\end{align}
	where $C_{n,\delta}$ is given later. 
	Under some {additional} conditions, with probability at least $1-7\delta$, we have
	\begin{align}
	\label{i:main}
	\|\Sigma^{\frac{1}{2}}(\beta^*-\hat{\beta})\|_2 \leq C_{\delta,\kappa,\rho,\sigma} r_{n,d,s,o},
	\end{align}
	where $C_{\delta,\kappa,\rho,\sigma}$ is a constant depending on $\delta, \kappa,\rho,\sigma,$ and 
	\begin{align*}
	r_{n,d,s,o} = \sqrt{\frac{s\log d}{n}} +\frac{o}{n} \sqrt{\log \frac{n}{o} \log n}.
	\end{align*}
\end{theorem}
\begin{remark}
In Section~\ref{sec:put-together}, first, we will obtain a more general estimation bound than \eqref{i:main} under a weaker condition than \eqref{e:par-pre}, given by
	\begin{align}
	\label{e:par-pre2}
	\lambda_o \ge C_{\lambda_o}\sqrt{\frac{2\sigma^2\log (n/\delta)}{n}},\quad
	\lambda_s \ge  \frac{4\sqrt{2}}{\sqrt{3}} C_{\lambda_s} \lambda_o. 
	\end{align}
Hereafter,  we basically assume \eqref{e:par-pre2} instead of \eqref{e:par-pre}. 
\end{remark}
\begin{corollary} 
	Assume the conditions used in Theorem~\ref{t:main-pre}. Suppose that $\lambda_{\min}$ is positive. Then, the inequality (\ref{i:main}) of Theorem~\ref{t:main-pre} implies 
	\begin{align}
	\label{c:main}
	\|\beta^*-\hat{\beta}\|_2 \leq \frac{C_{\delta,\kappa,\rho,\sigma}}{\lambda_{\min}} r_{n,d,s,o}.
	\end{align}	
\end{corollary}

Here we compare the above with the convergence rate of \cite{NguTra2012Robust}.

\begin{theorem}[\cite{NguTra2012Robust}]
	\label{t:NguTra2012Robust}
	Assume that $\lambda_{\min}$ is positive and $\lambda_{\max}$ satisfies $\lambda_{\max} \rho^2 = O(1)$. 
	Suppose that $n$ is sufficiently large with $n \geq c\frac{\rho^2}{\lambda_{\min}} s \log d$ and $o$ is sufficiently small with $o \leq \min \left(c_1 \frac{n}{\gamma \log n}, c_2 n\right)$, where $c,c_1,c_2$ are some constants and $\gamma \in (0,1]$.
	Let 
	\begin{align*}
	\lambda_o =  2 \sqrt{\frac{2\sigma^2 \log n}{n}},\quad \lambda_s = \frac{2}{\gamma} \sqrt{\frac{2\sigma^2\rho^2 \log d}{n}}\left(1+ \sqrt{\frac{2\log d}{n}}\right).
	\end{align*}
	With probability at least $1-c_3e^{-c_4n}$, we have 
	\begin{align*}
	\|\beta^*-\hat{\beta}\|_2 + \|\theta^*-\hat{\theta}\|_2 \leq {C_{\kappa, \rho, \sigma}} \left(\sqrt{\frac{s \log d}{n}} + \sqrt{\frac{o}{n} \log n }\right),
	\end{align*}
	where $c_3$ and $c_4$ are some constants and {$C_{\kappa, \rho, \sigma}$} is a constant depending on $\kappa,\rho,\sigma$.
\end{theorem}

A remarkable difference between the above two theorems is the convergence rate of the estimation error on the second term related to the number of outliers,~$o$.
The main theorem shows a faster convergence rate than Theorem~\ref{t:NguTra2012Robust}. This arises from careful analysis with a different setting of tuning parameters in the main theorem.

A large difference of conditions between the above two theorems is that the parameter $\delta$ does not appear in the tuning parameters 
in Theorem~\ref{t:NguTra2012Robust}, but the parameter $\delta$ is incorporated into the tuning parameters 
in the main theorem. 
The condition (\ref{cond0-pre}) may be complicated, however it is satisfied under some condition, including that $n$ is sufficiently large, as seen in Appendix~\ref{sec:cond0}. 

\subsection{Outline of the proof of the main theorem}
\label{s:robust-lasso-3}
Let $L(\beta)$ be the loss function with the parameter $\beta$. When we obtain the convergence rate of the estimation error of Lasso, we usually start with the inequality that $L(\hat{\beta}) \le L(\beta^*)$, where $\hat{\beta}$ is the minimizer of $L(\beta)$ and $\beta^*$ is the true value.
In this paper, we {employ} this approach {to obtain a simple but weak relation between} $\|\beta^*-\hat{\beta}\|_1$ and $\|\Sigma^{\frac{1}{2}}(\beta^*-\hat{\beta})\|_2$ in Section~\ref{ss:relation}, by virtue of the restricted eigenvalue condition. Although we partly use this approach, we mainly adopt a different approach to obtain a faster convergence rate. 
First, we focus on the fact that the derivative of the loss function is zero at the minimizer. Next, we divide the derivative of the loss function into three parts via a specific property of the Huber function. These are shown in this subsection. Some properties related to the three terms are given in the subsequent sections. Combining the results, we can show the main theorem. 

As shown in \cite{SheOwe2011Outlier}, the optimization problem (\ref{o:robust-lasso}) becomes 
\begin{align}
\label{o:robust-lasso-3}
\hat{\beta} \in {\argmin}_\beta L(\beta), \qquad L(\beta)=\lambda_o^2 \sum_{i=1}^n H\left(\frac{y_i-X_i^\top\beta}{\lambda_o\sqrt{n}}\right)+\lambda_s\|\beta\|_1,
\end{align}
where 
$H(t)$ is the Huber loss function, given by
\begin{align*}
H(t) = \begin{cases}
|t| -1/2 & (|t| > 1) \\
t^2/2  & (|t| \leq 1)
\end{cases}
.
\end{align*}
It should be noted that this is not a standard penalized Huber loss function, because the tuning parameter $\lambda_o$ is included in the Huber loss function. 
Let
$$ r_i(\beta)=\frac{y_i-X_i^\top\beta}{\lambda_o\sqrt{n}}. $$ 
Since the derivative of $L(\beta)$ about $\beta$ is zero at $\beta=\hat{\beta}$, we have
\begin{align}
\label{e:outline-huber}
\lambda_o \sum_{i=1}^n \frac{X_i}{\sqrt{n}} \hat{\psi}_i = \lambda_s{ \partial \|\hat{\beta}\|_1}, 
\end{align}
where $\psi(t)=H'(t)$, $\hat{\psi}_i=\psi \left(r_i(\hat{\beta})\right)${, $\partial \|\cdot\|_1$ is a subdifferential of $\| \cdot \|_1$}.
By multiplying $(\beta^*-\hat{\beta})^\top$ to the both side of the above equation, we have
\begin{align}
\lambda_o \sum_{i=1}^n (\beta^*-\hat{\beta})^\top \frac{X_i}{\sqrt{n}} \hat{\psi}_i = \lambda_s  (\beta^*-\hat{\beta})^\top {\partial\|\hat{\beta}\|_1} \le \lambda_s \| \beta^*-\hat{\beta} \|_1. \label{eq:basic}
\end{align}

Let $I_u$ and $I_o$ be the index sets for uncontaminated and contaminated outputs, respectively. Let
\begin{align}
I_> =  \left\{ i \in I_u: \ \left| r_i(\hat{\beta}) \right| >1 \right\}, \qquad 
{C_{cut} = \# \mr{I}_{>}}.
\end{align}
{These play important roles in the proof,} because the Huber loss function $H(r_i(\beta))$ changes the behavior according to whether $|r_i(\beta)|$ is larger than one or not. Let $ I_<=\{i \in I_u: \ | r_i(\hat{\beta}) | \le 1\} =I_u-I_>$. We see $\# I_u=n-o$, $\# I_o=o$ and $\# I_<=n-o-C_{cut}$. 
The left-hand side (L.H.S.) of \eqref{eq:basic} can be divided into three parts, given by 
\begin{align}
T_1+T_2+T_3  \le \lambda_s \| \beta^*-\hat{\beta} \|_1, \label{eq:basic2}
\end{align}
where $T_1$, $T_2$ and $T_3$ correspond to the index sets $I_<$, $I_>$ and $I_o$, respectively. 

As described later, each term of \eqref{eq:basic2} can be evaluated with a link to the estimation error
\begin{align*}
E=\|\Sigma^{\frac{1}{2}}(\beta^*-\hat{\beta})\|_2.
\end{align*}
The right-hand side (R.H.S.) of \eqref{eq:basic2} is evaluated in Section~\ref{sec:opt-re}. The value $C_{cut}$ is evaluated in Section~\ref{sec:huber-overflow}. The L.H.S of \eqref{eq:basic2}, including $C_{cut}$, is evaluated in Section~\ref{sec:dir}. Using these results, we can prove the main theorem.

\section{Preliminary}
\label{s:preliminary}

First, we present some properties related to the Gaussian width and Gaussian random matrix. 
Next, 
we prepare some concentration inequalities. 
These are the results shown by \cite{DalTho2019Outlier} and others. 
Finally, we summarize the conditions used in this section. 

First, we state the definition of the Gaussian width and show {three} properties of the Gaussian width.

\begin{definition}[Gaussian width]
	For a subset $T \subset \mbb{R}^d$, the Gaussian width is defined by
	\begin{align*}
	\mc{G}(T) := \mr{E} \sup_{x \in T} \langle g,x \rangle,
	\end{align*}
	where $g \sim \mc{N}(0,I_d)$.
\end{definition}

\begin{lemma}[Theorem 2.5 of \cite{BouLugMas2013concentration}]
	\label{l:gau-wid-d}
	We have
	\begin{align*}
	\mc{G}(\Sigma^{\frac{1}{2}}\mbb{B}_1^d) \leq  \sqrt{2\rho^2\log d}.
	\end{align*}
\end{lemma}

Let $\mbb{B}_1^m = \left\{u \in \mbb{R}^m:\|u\|_1 \leq 1\right\}$ and $\mbb{B}_2^m = \left\{u \in \mbb{R}^m:\|u\|_2 \leq 1\right\}$.
\begin{lemma}
	\label{l:gau-wid}
	For any vector $u \in \mbb{R}^n$, we have
	\begin{align}
	\label{i:gau-wid}
	\mc{G}(\|u\|_1\mbb{B}_1^n \cap \|u\|_2 \mbb{B}_2^n) \leq \|u\|_1 \sqrt{2\log n}.
	\end{align}
\end{lemma}
\begin{proof} 
	\begin{align*}
	\mc{G}(\|u\|_1\mbb{B}_1^n \cap \|u\|_2 \mbb{B}_2^n) \leq \mc{G}(\|u\|_1\mbb{B}_1^n ) = \|u\|_1 \mc{G}(\mbb{B}_1^n) \leq  \|u\|_1 \sqrt{2\log n}.
	\end{align*}
	The last inequality follow from  Lemma~\ref{l:gau-wid-d}.
\end{proof}

\begin{lemma}
	\label{l:gau-wid-cube}
	For any {$m$-sparse} vector $u \in \mbb{R}^n$, we have
	\begin{align}
	\label{i:gau-wid-cube}
	\mc{G}(\|u\|_1\mbb{B}_1^n \cap \|u\|_2 \mbb{B}_2^n) & \leq 4\sqrt{e} \sqrt{m}\sqrt{4+\log\frac{n}{m}}\|u\|_2.
	\end{align}
\end{lemma}
\begin{proof}
	Because $u$ is a $m$-sparse vector, $\|u\|_1 \leq \sqrt{m}\|u\|_2$ holds.
	Hence, setting $o=m$ on (39) in Remark 4 of \cite{DalTho2019Outlier}, we have
	\begin{align*}
	\mc{G}(\|u\|_1\mbb{B}_1^n \cap \|u\|_2 \mbb{B}_2^n) \leq 4\sqrt{em}\|u\|_2\sqrt{1+\log(8n/m)}.
	\end{align*}
	This inequality shows the inequality (\ref{i:gau-wid-cube}) from $1+\log 8 <4$.
\end{proof}
Here, we introduce two concentration inequalities of Gaussian random matrix. 

\begin{proposition}[Proposition 3 of \cite{DalTho2019Outlier}]
	\label{p:DT19-3}
	Let $Z\in\mbb{R}^{n\times p}$ be a random matrix with $i.i.d.\ \mc{N}(0,\Sigma)$ rows. For any $\delta \in (0,1/7]$ and $n \geq 100$, the following property holds with probability at least $(1-\delta)$: for any $v \in \mbb{R}^d$,
	\begin{align*}
	\left\|\frac{Z}{\sqrt{n}}v \right\|_2 &\geq a_1 \|\Sigma^{\frac{1}{2}} v\|_2 -\frac{1.2\mc{G}(\Sigma^{\frac{1}{2}} \mbb{B}_1^d)}{\sqrt{n}}\|v\|_1,
	\end{align*}
where
\begin{align*}{
a_1= 1-\frac{4.3+\sqrt{2 \log(9/\delta)}}{\sqrt{n}}.
}\end{align*}
\end{proposition}

\begin{proposition}[Proposition 4 of \cite{DalTho2019Outlier}]
	\label{p:DT19-4}
	Let $Z\in\mbb{R}^{n\times p}$ be a random matrix with $i.i.d.\ \mc{N}(0,\Sigma)$ rows. For any $\delta \in (0,1/7]$ and $n \in \mbb{N}$, the following property holds with probability at least $(1-\delta)$: for any {$u \in \mbb{R}^n$ and $v \in \mbb{R}^d$},
	\begin{align*}
	\left|u^\top \frac{Z}{\sqrt{n}}v\right| &\leq \|\Sigma^{\frac{1}{2}} v\|_2\|u\|_2 \sqrt{\frac{2}{n}}\left(4.8+\sqrt{\log\frac{81}{\delta}}\right)+1.2\|v\|_1\|u\|_2\frac{\mc{G}(\Sigma^{\frac{1}{2}}\mbb{B}^p_1)}{n}\\
	&+1.2\|\Sigma^{\frac{1}{2}} v\|_2\frac{\mc{G}(\|u\|_1\mbb{B}^n_1 \cap \|u\|_2\mbb{B}^n_2)}{\sqrt{n}}.
	\end{align*}
\end{proposition}

By Lemmas~\ref{l:gau-wid-cube}~and~\ref{l:gau-wid-d} and Propositions~\ref{p:DT19-3} and~\ref{p:DT19-4}, we can easily show the following corollaries.
\begin{corollary}
	\label{c:DT19-3}
	Let $Z\in\mbb{R}^{n\times p}$ be a random matrix with $i.i.d.\ \mc{N}(0,\Sigma)$ rows. For any $\delta \in (0,1/7]$ and $n \geq 100$, the following property holds with probability at least $(1-\delta)$: for any $v \in \mbb{R}^d$,
	\begin{align*}
\left\|\frac{Z}{\sqrt{n}}v \right\|_2 \geq a_1\|\Sigma^{\frac{1}{2}} v\|_2 -1.2\sqrt{\frac{2\rho^2 \log d}{n}}\|v\|_1.
	\end{align*}
\end{corollary}
\begin{corollary}
	\label{c:DT19-4}
	Let $Z\in\mbb{R}^{n\times p}$ be a random matrix with $i.i.d.\ \mc{N}(0,\Sigma)$ rows. For any $\delta \in (0,1/7]$ and $n \in \mbb{N}$, the following property holds with probability at least $(1-\delta)$: for any $m$-sparse $u \in \mbb{R}^n$ and any $v \in \mbb{R}^d$,
	\begin{align*}
\left|u^\top\frac{Z}{\sqrt{n}}v\right|  &\leq \|\Sigma^{\frac{1}{2}} v\|_2\|u\|_2 {\sqrt{\frac{2}{{n}}} \left(4.8+\sqrt{\log \frac{81}{\delta}}\right)} +1.2\|v\|_1\|u\|_2\sqrt{\frac{2\rho^2 \log d}{n}} \\
		& \hspace*{10mm} +4.8\sqrt{e}\|\Sigma^{\frac{1}{2}} v\|_2\|u\|_2\sqrt{\frac{m}{n}}\sqrt{4+\log\frac{n}{m}}.
		\end{align*}
\end{corollary}

Next, we prepare three inequalities of concentration inequalities. 

\begin{proposition}[Lemma 2 of \cite{DalTho2019Outlier}]
	\label{p:DalTho2019Outlierlemma2}
	Let $\{\xi_i\}_{i=1}^n$ be a sequence with i.i.d random variables drawn from $\mc{N}(0,\sigma^2)$ and 
	$\{X_i\}_{i=1}^n$  drawn from $\mc{N}(0,\Sigma)$.
	Let $\xi = (\xi_1,\cdots,\xi_n)^\top$ and $X = [X_1,\cdots, X_n]^\top$. 
	For any $\delta \in (0,1)$ and $n \geq 2 \log (d/\delta)$, with probability at least $(1-\delta)^3$, we have
	\begin{align*}
	\left\|\frac{\xi}{\sqrt{n}}\right\|_\infty \leq \sqrt{\frac{2\sigma^2 \log (n/\delta)}{n}},\\ 
\left\|\frac{X^\top \xi}{n}\right\|_\infty \leq {2} \sqrt{\frac{2 \sigma^2 \rho^2 \log (d/\delta)}{n}}.
	\end{align*}
\end{proposition}

\begin{proposition}
	\label{p:XxiBernstein}
	Let $\{\xi_i\}_{i=1}^n$ be a sequence with i.i.d random variables drawn from $\mc{N}(0,\sigma^2)$ and 
	$\{X_i\}_{i=1}^n$  drawn from $\mc{N}(0,\Sigma)$. 
	Let $z_{ij}= X_{ij} \psi\left(\frac{\xi_i}{\lambda_o\sqrt{n}}\right)$ and $z = (\sum_{i=1}^nz_{i1}, \cdots ,\sum_{i=1}^nz_{d,1})$.
	For any $\delta \in (0,1)$ and $n$ such that $\sqrt{\frac{\log (d/ \delta)}{n}} \leq \sqrt{3} -\sqrt{2}$, with probability at least $1-\delta$, we have
		\begin{align*}
		\left\|\frac{z}{\sqrt{n}} \right\|_\infty \leq \sqrt{3\frac{\rho^2 \sigma^2}{n\lambda_o^2}\log \frac{d}{\delta}}=:C_z.
		\end{align*}
\end{proposition}
{The proof of Proposition~\ref{p:XxiBernstein} is given in Appendix A. }

\begin{proposition}[Lemma~1 of \cite{LauMas2000Adaptive}]
	\label{p:LauMas2000Adaptivelemma1}
	Let $\{\xi_i\}_{i=1}^n$ be a sequence with i.i.d random variables drawn from $\mc{N}(0,\sigma^2)$.
	For any $\delta \in (0,1)$ and $n$ such that $2\sqrt{n \log (1/\delta)} + 2\log(1/\delta) \leq n$, with probability at least $1-\delta$, we have
	\begin{align*}
	\frac{1}{n} \sum_{i=1}^n \xi_i^2\leq 2\sigma^2.
	\end{align*}
\end{proposition}

Finally, we summarize the conditions used above:
\begin{align*}
{\rm(c1)} & \  \ \delta \in (0,1/7], \, n \geq 100. \\
{\rm (c2)} & \ \ \sqrt{\frac{\log(d/\delta)}{ n}} \leq \sqrt{3} -\sqrt{2}.\quad (\text{This implies } 2\log (d/\delta) \leq n{\text{ with } n \geq 100 }). \\
{\rm (c3)} & \ 2\sqrt{n \log (1/\delta)} + 2\log(1/\delta) \leq n.
\end{align*}

Based on the results prepared in this section, we will show many propositions and finally prove the main theorem. Hereafter, we will use the phrase "with a high probability" without explicit probability in the propositions. We give an explicit probability in Section~\ref{sec:Outline_Proof_main}.


\section{Relation between $\|\beta^* -\hat{\beta}\|_1$ and $\|\Sigma^{\frac{1}{2}}(\beta^*-\hat{\beta})\|_2$}
\label{sec:opt-re}

As seen in {Section~\ref{s:outline},} a relation between $\|\beta^* -\hat{\beta}\|_1$ and $\|\Sigma^{\frac{1}{2}}(\beta^*-\hat{\beta})\|_2$ plays an important role to obtain the convergence rate of the estimation error.  

First, we introduce a restricted eigenvalue condition and a simple lemma, which are often used to obtain the convergence rate of the estimation error. Next, we obtain a relation between $\|\beta^* -\hat{\beta}\|_1$ and $\|\Sigma^{\frac{1}{2}}(\beta^*-\hat{\beta})\|_2$.

\subsection{Restricted eigenvalue condition}

For a set $J$, let {$\# J$} represent the number of elements of $J$. 
For a vector $v=(v_1,\ldots,v_d)^\top \in \mbb{R}^d$ and a set $J \subset \{1,\ldots,d\}$, let $v_J$ be the vector whose $j$th component is $v_j$ for $j \in J$ and $0$ for $j \notin J$.

The restricted eigenvalue condition for $\Sigma$ is defined in the following and this condition enable us to deal with the case where $\Sigma$ is singular. 
\begin{definition}[Restricted Eigenvalue Condition \cite{DalTho2019Outlier} ] \label{def:RE}
	The matrix $\Sigma$ is said to satisfy the restricted eigenvalue condition $\mr{RE}(s,c_0,\kappa)$ with a positive integer $s$ and positive values $c_0$ and $\kappa$, if
	\begin{align*}
	\kappa \|v_J\|_2 \leq \|\Sigma^{\frac{1}{2}} v\|_2
	\end{align*}
	for any set $J \subset \{1,\cdots,d\}$ and any vector $v \in \mbb{R}^d$ such that $|J| \leq s$,
	\begin{align}
	\label{con:re-v}
	\|v_{J^c}\|_1\leq c_0\|v_J\|_1.
	\end{align}
\end{definition}

When the matrix $\Sigma$ satisfies the restricted eigenvalue condition $\mr{RE}(s,c_0,\kappa)$, we immediately obtain the following lemma. 

\begin{lemma} 
	\label{l:rev}
	Suppose that $\Sigma$ satisfies $\mr{RE}(s,c_0,\kappa)$. Then, we have
	\begin{align}
	\label{i:re-norm-1}
	\|v\|_1 &\leq \frac{c_0+1}{\kappa}\sqrt{s}\|\Sigma^{\frac{1}{2}} v\|_2,
	\end{align}
	for any $v \in \mbb{R}^d$ satisfying (\ref{con:re-v}) for every $J \subset \{1,\cdots,d\}$ with $|J| \leq s$.
\end{lemma}
\begin{proof}
	\begin{align*}
	\|v\|_1 = \|v_J\|_1+\|v_{J^c}\|_1 &\leq (c_0+1)\|v_J\|_1 
	=\frac{c_0+1}{\kappa} \sqrt{s}\kappa\|v_J\|_2 \leq \frac{c_0+1}{\kappa} \sqrt{s}\|\Sigma^{\frac{1}{2}} v\|_2.
	\end{align*}
\end{proof}

\subsection{Relation between $\|\beta^* -\hat{\beta}\|_1$ and $\|\Sigma^{\frac{1}{2}}(\beta^*-\hat{\beta})\|_2$}
\label{ss:relation}

The following proposition plays an important role to show a relation between $\|\beta^* -\hat{\beta}\|_1$ and $\|\Sigma^{\frac{1}{2}}(\beta^*-\hat{\beta})\|_2$. The proof is given in Section~\ref{sec:proofofprop}. 
Let the active set be denoted by $S=\left\{i:\beta^*_i  \neq 0\right\}$. 

\begin{proposition}
	\label{p:coe-re}
	Assume the conditions (c1) and (c2). 
	Suppose that $\lambda_s - C_{\lambda_s} \lambda_o>0$ and $\|\Sigma^{\frac{1}{2}}(\beta^*-\hat{\beta})\|_2 \leq \frac{1}{\sqrt{s}}\|\beta^*-\hat{\beta}\|_1$. Then, with a high probability, we have
	\begin{align*}
	\|\beta_{S^c}^*-\hat{\beta}_{S^c}\|_1  \leq \frac{\lambda_s + C_{\lambda_s} \lambda_o}{\lambda_s - C_{\lambda_s} \lambda_o } \|\beta^*_S-\hat{\beta}_S\|_1,
	\end{align*}
	where $C_{\lambda_s}$ is defined in Theorem~\ref{t:main-pre}.
\end{proposition}

Combining Lemma~\ref{l:rev} with Proposition~\ref{p:coe-re}, we can easily prove the following proposition, which shows a relation between $\|\beta^* -\hat{\beta}\|_1$ and $\|\Sigma^{\frac{1}{2}}(\beta^*-\hat{\beta})\|_2$. 

\begin{proposition}
	\label{p:coe-1-2-norm}
	Assume the conditions (c1) and (c2). 
	Suppose that $\Sigma$ satisfies $\mr{RE}(s,c_0,\kappa)$, $\lambda_s - C_{\lambda_s} \lambda_o>0$ and 
	\begin{align}
	\label{i:so}
	\frac{\lambda_s + C_{\lambda_s}\lambda_o}{\lambda_s - C_{\lambda_s} \lambda_o}  \leq  c_0.
	\end{align}
	Then, with a high probability, we have
	\begin{align}
	\label{e:l1l2}
	\|\beta^*-\hat{\beta}\|_1 \leq c_\kappa \sqrt{s}\|\Sigma^{\frac{1}{2}} (\beta^*-\hat{\beta})\|_2,
	\end{align}
	where $c_{\kappa}:=\frac{c_0+1}{\kappa}+1$.
\end{proposition}

\begin{proof}
	When $\|\beta^*-\hat{\beta}\|_1 < \sqrt{s}\|\Sigma^{\frac{1}{2}}(\beta^*-\hat{\beta})\|_2$, we have \eqref{e:l1l2} immediately {since $c_\kappa \geq 1$}.
	Consider the case where $\|\beta^*-\hat{\beta}\|_1 \geq \sqrt{s}\|\Sigma^{\frac{1}{2}}(\beta^*-\hat{\beta})\|_2$. Let $J=S$ and then we have $|J|=|S| \le s$. 
	From Proposition~\ref{p:coe-re} and condition \eqref{i:so}, $v=\beta^*-\hat{\beta}$ satisfies $\|v_{J^c}\|_1 \le c_0 \| v_J\|_1$, that is, the condition \eqref{con:re-v}. Hence, since $\Sigma$ satisfies ${\rm RE}(s,c_0,\kappa)$, we have the property \eqref{i:re-norm-1} with $v=\beta^*-\hat{\beta}$, so that we see $\|v\|_1 \le c_\kappa\sqrt{s}\|\Sigma^{\frac{1}{2}} v \|_2$ since $(c_0+1)/\kappa \le c_\kappa$, and then the property \eqref{e:l1l2} holds. 
\end{proof}

\subsection{Inequalities related to the estimation error}

Using Corollaries~\ref{c:DT19-3} and \ref{c:DT19-4} and Proposition~\ref{p:coe-1-2-norm}, we can easily show the following two propositions related to the estimation error $\|\Sigma^{\frac{1}{2}}(\beta^*-\hat{\beta})\|_2$. 

\begin{proposition}
	\label{p:DT19-3_2}
	Assume the conditions used in Proposition~\ref{p:coe-1-2-norm}. 
	Then,  with a high probability, we have
	\begin{align*}
	\left\|\frac{X}{\sqrt{n}}(\beta^*-\hat{\beta})\right\|_2 \geq C_\kappa \|\Sigma^{\frac{1}{2}}(\beta^*-\hat{\beta})\|_2, \qquad C_\kappa=  a_1 -1.2c_\kappa\sqrt{\frac{2\rho^2 s\log d}{n}}.
	\end{align*}
\end{proposition}
\begin{proof}
	By letting $v =\beta^*-\hat{\beta}$ in Corollary~\ref{c:DT19-3}, we have 
	\begin{align*}
	\left\|\frac{X}{\sqrt{n}}(\beta^*-\hat{\beta})\right\|_2&\geq a_1 \|\Sigma^{\frac{1}{2}}(\beta^*-\hat{\beta})\|_2 -1.2\sqrt{\frac{2\rho^2 \log d}{n}}\|\beta^*-\hat{\beta}\|_1.
	\end{align*}
The proof is complete from \eqref{e:l1l2} in Proposition~\ref{p:coe-1-2-norm}
\end{proof}

\begin{proposition}
	\label{p:DT19-4_2}
	Assume the conditions used in Proposition~\ref{p:coe-1-2-norm}.
	Then, the following property holds with a high probability: for any $m$-sparse vector $u \in \mbb{R}^n$,
	\begin{equation*}
	\left|u^\top\frac{X}{\sqrt{n}}(\beta^*-\hat{\beta})\right| \leq  \|\Sigma^{\frac{1}{2}} (\beta^*-\hat{\beta})\|_2\|u\|_2 g(m).
	\end{equation*}
\end{proposition}
\begin{proof}
	By letting $v=\beta^*-\hat{\beta}$ in Corollary~\ref{c:DT19-4}, we have 
	\begin{eqnarray*}
		\lefteqn{ \left|u^\top \frac{X}{\sqrt{n}}(\beta^*-\hat{\beta})\right| }\\
		& \leq&  \|\Sigma^{\frac{1}{2}}(\beta^*-\hat{\beta})\|_2\|u\|_2\sqrt{\frac{2}{n}} \left( 4.8+ \sqrt{\log \frac{81}{\delta}}\right) + 1.2\|\beta^*-\hat{\beta}\|_1\|u\|_2 \sqrt{\frac{2\rho^2 \log d}{n}}\\
		& & + 4.8\sqrt{e}\|\Sigma^{\frac{1}{2}}(\beta^*-\hat{\beta})\|_2 \|u\|_2 \sqrt{\frac{m}{n}} \sqrt{4+\log \frac{n}{m}}. 
	\end{eqnarray*}
The proof is complete from \eqref{e:l1l2} in Proposition~\ref{p:coe-1-2-norm}.
\end{proof}


\subsection{Proof of Proposition~\ref{p:coe-re}}
\label{sec:proofofprop}
Since $\hat{\beta}$ is the minimizer of $L(\beta)$, we have $L(\hat{\beta}) \leq L(\beta^*)$, which implies
\begin{align} \label{eq:fuji-4-4-1}
\lambda_o^2 \sum_{i=1}^n H(r_i(\hat{\beta}))- \lambda_o^2 \sum_{i=1}^n H(r_i(\beta^*)) \leq \lambda_s(\|\beta^*\|_1-\|\hat{\beta}\|_1) .
\end{align}
Since $H(z)$ is convex, we have 
\begin{align} \label{eq:fuji-4-4-2}
H(r_i(\hat{\beta})) - H(r_i(\beta^*)) 
\ge \psi(r_i(\beta^*)) \{ r_i(\hat{\beta}) - r_i(\beta^*)\} =
\psi(r_i(\beta^*))\frac{X_i^\top (\beta^*-\hat{\beta})}{\lambda_o\sqrt{n}}
\end{align}
and then  
\begin{align} \label{eq:fuji-4-1}
\frac{\lambda_o}{\sqrt{n}} \sum_{i=1}^n \psi(r_i(\beta^*)) X_i^\top (\beta^*-\hat{\beta})  \leq \lambda_s(\|\beta^*\|_1-\|\hat{\beta}\|_1).
\end{align}
Since {$Y_i=X_i^\top \beta^*+\xi_i$} for $i \in I_u$ and {$Y_i=X_i^\top\beta^*+\sqrt{n}\theta_i + \xi_i$} for $i \in I_o$, 
we have
\begin{align*}
r_i(\beta^*) =
\begin{cases}
\ {\displaystyle \frac{\xi_i}{\lambda_o\sqrt{n}} } &  (i \in I_u) \\[4mm]
\ {\displaystyle \frac{\sqrt{n}\theta_i+\xi_i}{\lambda_o\sqrt{n}} } & (i \in I_o)
\end{cases}
.
\end{align*}
We divide the summation in \eqref{eq:fuji-4-1} into two parts:
\begin{align}
\label{i:min-opt}
\lambda_s(\|\beta^*\|_1-\|\hat{\beta}\|_1)  &\geq \frac{\lambda_o}{\sqrt{n}} \sum_{i \in I_u} \psi\left( \frac{\xi_i}{\lambda_o\sqrt{n}} \right)X_i^\top (\beta^*-\hat{\beta})  \nonumber \\
&  + \frac{\lambda_o}{\sqrt{n}} \sum_{i \in I_o} \psi \left( \frac{\sqrt{n}\theta_i+\xi_i}{\lambda_o\sqrt{n}} \right) X_i^\top (\beta^*-\hat{\beta})  \\ \nonumber
&= \frac{\lambda_o}{\sqrt{n}} \sum_{i=1}^n \psi\left( \frac{\xi_i}{\lambda_o\sqrt{n}} \right) X_i ^\top(\beta^*-\hat{\beta})  + \frac{\lambda_o}{\sqrt{n}} \sum_{i=1}^n u_i X_i ^\top(\beta^*-\hat{\beta}),
\end{align}
where
\begin{align*}
u_i =
\begin{cases}
\ 0 & (i \in I_u) \\
\ {\displaystyle \psi \left( \frac{\sqrt{n}\theta_i+\xi_i}{\lambda_o\sqrt{n}} \right) - \psi\left( \frac{\xi_i}{\lambda_o\sqrt{n}} \right)  } & (i \in I_o)
\end{cases}
.
\end{align*}
The first term can be evaluated by a standard technique, because it includes only standard quantities. 
The second term is difficult to be evaluated, because it includes the terms related to adversarial outliers, $\sqrt{n}\theta_i$s. It can be evaluated by virtue of Corollary~\ref{c:DT19-4}, {because} $\# I_o=o$ and $u=(u_1,\ldots,u_n)^\top$ is an $o$-sparse vector. 
Remember the notations $z_{ij}=X_{ij} \psi\left(\frac{\xi_i}{\lambda_o\sqrt{n}}\right)$ and $z = \left( \sum_{i=1}^n z_{i1},\ldots,\sum_{i=1}^n z_{id}\right)^\top $.  From (\ref{i:min-opt}), we see
\begin{align}
\label{eq:fuji-4-2}
0 
&\leq - \frac{\lambda_o}{\sqrt{n}} z^\top(\beta^*-\hat{\beta}) - \lambda_o u^\top \frac{X}{\sqrt{n}} (\beta^*-\hat{\beta}) + \lambda_s (\|\beta^*\|_1 -\|\hat{\beta}\|_1) \nonumber \\
&\leq \frac{\lambda_o}{\sqrt{n}} \left| z^\top(\beta^*-\hat{\beta}) \right| + 
\lambda_o \left| u^\top \frac{X}{\sqrt{n}}(\beta^*-\hat{\beta})\right| + \lambda_s (\|\beta^*\|_1 -\|\hat{\beta}\|_1) \nonumber \\
& \le  \frac{\lambda_o}{\sqrt{n}} \left\|z \right\|_\infty \|\beta^*-\hat{\beta}^*\|_1+ \lambda_o\left|u^\top \frac{X}{\sqrt{n}}(\beta^*-\hat{\beta})\right| + \lambda_s (\|\beta^*\|_1 -\|\hat{\beta}\|_1). 
\end{align}

First, we consider the first term of \eqref{eq:fuji-4-2}, which can be evaluated from Proposition~\ref{p:XxiBernstein}. 
Next, we consider the second term of \eqref{eq:fuji-4-2}. Since $|\psi(t)| \le 1$, we have $|u_i| \le 2$ for $i \in I_o$ and $\|u\|_2 \le \sqrt{2o}$ 
. Since $u$ is an $o$-sparse vector, it holds from Corollary~\ref{c:DT19-4} and the assumption $\|\Sigma^{\frac{1}{2}}(\beta^*-\hat{\beta})\|_2 \leq \frac{1}{\sqrt{s}}\|\beta^*-\hat{\beta}\|_1$ that 
\begin{eqnarray}
\lefteqn{ \left|u^\top \frac{X}{\sqrt{n}}(\beta^*-\hat{\beta})\right| } \nonumber \\
&\leq&  \|u\|_2 \left\{ \|\Sigma^{\frac{1}{2}} (\beta^*-\hat{\beta})\|_2  \sqrt{\frac{2}{n}} \left( 4.8+ \sqrt{\log \frac{81}{\delta}} \right) + 1.2  \|\beta^*-\hat{\beta}\|_1 \sqrt{\frac{2\rho^2 \log d}{n}} \right. \nonumber \\
& & \left. + 4.8\sqrt{e} \|\Sigma^{\frac{1}{2}} (\beta^*-\hat{\beta})\|_2 \sqrt{\frac{o}{n}} \sqrt{4+\log\frac{n}{o}}\right\} \nonumber \\
&\leq&  \|\beta^*-\hat{\beta}\|_1 C_s, \label{eq:fuji-4-4}
\end{eqnarray}
where $C_s=\sqrt{2o/s} g(o)$. 
Applying \eqref{eq:fuji-4-4} and Proposition~\ref{p:XxiBernstein} to \eqref{eq:fuji-4-2}, we have
\begin{align*}
0 &\leq \lambda_o \|\beta^*-\hat{\beta}\|_1 C_z+ \lambda_o\|\beta^*-\hat{\beta}\|_1 C_s+\lambda_s(\|\beta^*\|_1-\|\hat{\beta}\|_1)\\
& = \left(C_z+C_s\right) \lambda_o \left(\|\beta^*_S-\hat{\beta}_S\|_1+\|\beta^*_{S^c}-\hat{\beta}_{S^c}\|_1\right) + \lambda_s(\|\beta^*_{S}\|_1+\|\beta^*_{S^c}\|_1-\|\hat{\beta}_{S}\|_1-\|\hat{\beta}_{S^c}\|_1)\\
& =  C_{\lambda_s}\lambda_o\left(\|\beta^*_S-\hat{\beta}_S\|_1+\|\hat{\beta}_{S^c}\|_1\right) + \lambda_s(\|\beta^*_{S}\|_1-\|\hat{\beta}_{S}\|_1-\|\hat{\beta}_{S^c}\|_1)\\
& \leq  C_{\lambda_s}\lambda_o \left(\|\beta^*_S-\hat{\beta}_S\|_1+\|\hat{\beta}_{S^c}\|_1\right) + \lambda_s(\|\beta^*_{S}-\hat{\beta}_{S}\|_1-\|\hat{\beta}_{S^c}\|_1)\\
& =\left(\lambda_s +  C_{\lambda_s} \lambda_o \right) \|\beta^*_S-\hat{\beta}_S\|_1
+\left(-\lambda_s+  C_{\lambda_s}\lambda_o \right)\|\hat{\beta}_{S^c}\|_1.
\end{align*}
The proof is complete.

\section{Evaluation of $C_{cut}$}
\label{sec:huber-overflow}

As seen in Section~\ref{s:robust-lasso-3}, the integer $C_{cut}$ plays an important role in the proofs. In this section, we give an upper bound of $C_{cut}$. 

First, we give two lemmas.
\begin{lemma} \label{l:C_{cut}-app-pre1} 
	Assume  the condition (c1) and (c2). Then, with a high probability, we have
	\begin{align*}
	&\frac{C_{cut}}{2} \leq \sum_{i \in I_>}H\left(r_i(\hat{\beta})\right), \\
	&\sum_{i \in I_>}H\left(r_i(\beta^*)\right)
	\leq \frac{C_{cut} \sigma^2 \log (n/\delta)}{\lambda_o^2 n}. 
	\end{align*}
\end{lemma}
\begin{proof}
	For $i \in I_>$, we have $|r_i(\hat{\beta})|>1$. The Huber function satisfies $H(t) = |t|-1/2$ for $|t|>1$, so that $H(r_i(\hat{\beta}))>1/2$ for $i \in I_>$, which shows the first inequality since $\#I_>=C_{cut}$. 
	The Huber function satisfies $H(t) \le t^2/2$. We have $r_i(\beta^*)=\xi_i/\lambda_o\sqrt{n}$ for $i \in I_> \subset I_u$. 
	Proposition~\ref{p:DalTho2019Outlierlemma2} holds from the conditions (c1) and (c2), so that we have $\xi_i^2/n \leq 2\sigma^2 \log (n/\delta)/n$.
	Combining these results, we see
	$$ \sum_{i \in I_>}H\left(r_i(\beta^*)\right)
	\le \sum_{i \in I_>}\frac{r_i(\beta^*)^2}{2} 
	= \sum_{i \in I_>} \frac{\xi_i^2}{2\lambda_o^2n} \le \frac{C_{cut} \sigma^2 \log (n/\delta)}{\lambda_o^2n}, $$
	since $\# I_>=C_{cut}$, which shows the second inequality.
\end{proof}

\begin{lemma} \label{l:C_{cut}-app-pre2}  
	Assume {(c3)} and the conditions used in Proposition~\ref{p:coe-1-2-norm}.
	Then, with a high probability, we have
	\begin{align*}
	\left|\sum_{i \in I_o} \psi\left(r_i(\beta^*)\right)\frac{X_i(\beta^*-\hat{\beta})}{\lambda_o\sqrt{n}} \right|
	\leq   \frac{\sqrt{o}}{\lambda_o}g(o) \|\Sigma^{\frac{1}{2}}(\beta^*-\hat{\beta})\|_2, \\
	\left|\sum_{i \in I_<} \psi\left(r_i(\beta^*)\right)\frac{X_i(\beta^*-\hat{\beta})}{\lambda_o\sqrt{n}} \right|
	\leq \frac{\sqrt{2\sigma^2}}{\lambda_o^2} g(n-o)\|\Sigma^{\frac{1}{2}}(\beta^*-\hat{\beta})\|_2.
	\end{align*}
\end{lemma}
\begin{proof}
	Let 
	$$ 
	u_i=
	\begin{cases}
	\ \psi\left(r_i(\beta^*)\right) & (i \in I_o) \\
	\ 0 & (i \not\in I_o)
	\end{cases}
	.
	$$
	{We} see that $u$ is an $o$-sparse vector and $\|u\|_2 \le \sqrt{o}$ since $|\psi(t)| \le 1$ and $\# I_o=o$. From Proposition~\ref{p:DT19-4_2}, 
	\begin{align*}
	\left|\sum_{i \in I_o} \psi\left(r_i(\beta^*)\right)\frac{X_i(\beta^*-\hat{\beta})}{\lambda_o\sqrt{n}} \right| 
	&=  \frac{1}{\lambda_o} \left| u^\top \frac{X}{\sqrt{n}} (\beta^*-\hat{\beta}) \right| \ \leq \frac{1}{\lambda_o}\|\Sigma^{\frac{1}{2}} (\beta^*-\hat{\beta})\|_2\|u\|_2 g(o) \\
	&\leq \frac{\sqrt{o}}{\lambda_o}\|\Sigma^{\frac{1}{2}} (\beta^*-\hat{\beta})\|_2 g(o),
	\end{align*}
	which shows the first inequality of the proposition. 
	Let $u=(u_1,\ldots,u_n)^\top$ be redefined by
	$$ 
	u_i=
	\begin{cases}
	\ \psi\left(r_i(\beta^*)\right) & (i \in I_<) \\
	\ 0 & (i \not\in I_<)
	\end{cases}
	.
	$$
	From a similar discussion to the above, we have
	\begin{align*}
	\left|\sum_{i \in I_<} \psi\left(r_i(\beta^*)\right)\frac{X_i(\beta^*-\hat{\beta})}{\lambda_o\sqrt{n}} \right| & \leq  \frac{1}{\lambda_o}\|\Sigma^{\frac{1}{2}} (\beta^*-\hat{\beta})\|_2\|u\|_2 g(n-C_{cut}-o)\\
	&  \leq  \frac{1}{\lambda_o}\|\Sigma^{\frac{1}{2}} (\beta^*-\hat{\beta})\|_2\|u\|_2 g(n-o),
	\end{align*}
	using the monotonicity of $g(m)$. {Proposition~\ref{p:LauMas2000Adaptivelemma1} holds from the condition (c3).}
	By $\psi(u) \le |u|$ and Proposition~\ref{p:LauMas2000Adaptivelemma1}, 
	\begin{align*}
	\|u\|_2 = \sqrt{\sum_{i \in I_<} \psi(r_i(\beta^*))^2} 
	\leq \sqrt{\sum_{i \in I_<} r_i(\beta^*)^2 } {=} \sqrt{\sum_{i=1}^{n} \left( \frac{\xi_i}{\lambda_o \sqrt{n}}\right)^2 } \le \frac{\sqrt{2\sigma^2}}{\lambda_o}.
	\end{align*}
	The above two inequalities show the second inequality of the proposition. 
\end{proof}

Using the above two lemmas, we give an upper bound of $C_{cut}$. 

\begin{proposition}
	\label{p:C/n}
	Assume (c3) and the conditions used in Proposition~\ref{p:coe-1-2-norm}.
	{Suppose that} $\Sigma$ satisfies $\mr{RE}(s,c_0,\kappa)$ and $C_{\lambda_o}>1$.
	Then, with a high probability, we have
	\begin{align*} 
	C_{cut} \leq \frac{2C_r}{\lambda_o^2} \left( \sqrt{2\sigma^2} g(n-o) + \sqrt{o}{\lambda_o}g(o) +\sqrt{s}c_\kappa\lambda_s \right)\|\Sigma^{\frac{1}{2}}(\beta^*-\hat{\beta})\|_2.
	\end{align*}
where $C_r=1/(1-\frac{2\sigma^2 \log (n/\delta)}{\lambda_o^2n}) >0$.
\end{proposition}
\begin{proof} 
	From \eqref{eq:fuji-4-4-1}, we have
	\begin{align}
	\lambda_o^2 \sum_{i=1}^n \left\{ H(r_i(\hat{\beta})) - H\left(r_i(\beta^*)\right) \right\} \le \lambda_s \left( \|\beta^*\|_1 - \|\hat{\beta}\|_1 \right) \le \lambda_s \|\hat{\beta}- \beta^*\|_1 .
	\end{align}
	From \eqref{eq:fuji-4-4-2}, we have
	\begin{align*}
	\lefteqn{ \lambda_s \|\hat{\beta}- \beta^*\|_1
		\ge \lambda_o^2 \left\{ \sum_{i \in I_>} + \sum_{i \in I_<} + \sum_{i \in I_o} \right\} \left\{ H(r_i(\hat{\beta})) - H\left(r_i(\beta^*)\right) \right\} } \\
	& \ge  \lambda_o^2 \sum_{i \in I_>} \left\{ H(r_i(\hat{\beta})) - H\left(r_i(\beta^*)\right) \right\}  + \lambda_o^2 \left\{ \sum_{i \in I_<} +\sum_{i \in I_o} \right\} \psi\left( r_i(\beta^*)\right) \frac{X_i^\top(\beta^*-\hat{\beta})}{\lambda_o\sqrt{n}}  \\
	& \ge \lambda_o^2 \sum_{i \in I_>} \left\{ H(r_i(\hat{\beta})) - H\left(r_i(\beta^*)\right) \right\} - \lambda_o^2 \left\{ \sum_{i \in I_<} +\sum_{i \in I_o} \right\} \left|  \psi\left( r_i(\beta^*)\right) \frac{X_i^\top(\beta^*-\hat{\beta})}{\lambda_o\sqrt{n}} \right|
	\end{align*}
    and 
	\begin{align*}
	& \lambda_o^2 \left\{ \sum_{i \in I_<} +\sum_{i \in I_o} \right\} \left|  \psi\left( r_i(\beta^*)\right) \frac{X_i^\top(\beta^*-\hat{\beta})}{\lambda_o\sqrt{n}} \right| +\lambda_s \|\hat{\beta}- \beta^*\|_1 \\
	& \hspace*{20mm} \ge \lambda_o^2 \sum_{i \in I_>} \left\{ H(r_i(\hat{\beta})) - H\left(r_i(\beta^*)\right) \right\}.
	\end{align*}
	From Lemmas~\ref{l:C_{cut}-app-pre1}~and~\ref{l:C_{cut}-app-pre2} and Proposition~\ref{p:coe-1-2-norm}, we have
	\begin{align*}
	&\|\Sigma^{\frac{1}{2}}(\beta^* - \hat{\beta})\|_2\left( \sqrt{2 \sigma^2} g(n-o)+\sqrt{o}\lambda_og(o) + \sqrt{s} c_\kappa  \lambda_s\right) \\
	& \hspace*{20mm} \ge C_{cut} \left(\frac{\lambda_o^2}{2}-\frac{\sigma^2\log(n/\delta)}{n}  \right) = C_{cut} \frac{\lambda_o^2}{2C_r}.
	\end{align*}
	From \eqref{e:par-pre2} with $C_{\lambda_o}>1$,  we have $C_r>0$ and then
	\begin{align*}
	C_{cut} \leq \frac{2C_r}{\lambda_o^2} \left( \sqrt{2\sigma^2} g(n-o) + \sqrt{o}{\lambda_o}g(o) +\sqrt{s}c_\kappa\lambda_s \right)\|\Sigma^{\frac{1}{2}}(\beta^*-\hat{\beta})\|_2.
	\end{align*}
\end{proof}

\section{Outline of the proofs of the key propositions}
\label{sec:dir}

In this section, we evaluate $T_1$, {$T_2$} and $T_3$ in a rough manner. Detailed evaluations are given in Section~\ref{sec:put-together}. 

Let $X_{I_<}$ be the $\# I_< \times d$ matrix whose row vectors consist of $X_i$s ($i \in I_<$). Let $X_{I_>}$ and $X_{I_o}$ be defined in a similar {manner}. 
Let $\xi_{I_<}$ be the $\# I_<$ dimensional vector whose components consist of $\xi_i$s ($i \in I_<$). Let $\xi_{I_>}$ and $\xi_{I_o}$ be defined in a similar manner. 

We see
\begin{align*}
T_2
&= \lambda_o \sum_{i \in I_>} \psi \left( r_i(\hat{\beta}) \right) \frac{X_i^\top}{\sqrt{n}} (\beta^*-\hat{\beta}) 
= \lambda_o \sum_{i \in I_>} \mr{sgn} (r_i(\hat{\beta}))\frac{X_i^\top}{\sqrt{n}} (\beta^*-\hat{\beta}), \\
T_3
&= \lambda_o \sum_{i \in I_o} \psi \left( r_i(\hat{\beta}) \right) \frac{X_i^\top}{\sqrt{n}} (\beta^*-\hat{\beta}), \\
T_1
&= \lambda_o \sum_{i \in I_<} \psi \left( r_i(\hat{\beta}) \right) \frac{X_i^\top}{\sqrt{n}} (\beta^*-\hat{\beta})
= \lambda_o \sum_{i \in I_<} r_i(\hat{\beta}) \frac{X_i^\top}{\sqrt{n}} (\beta^*-\hat{\beta}) \\
&= \lambda_o \sum_{i \in I_<} \frac{y_i-X_i^\top\hat{\beta}}{\lambda_o \sqrt{n}} \frac{X_i^\top}{\sqrt{n}} (\beta^*-\hat{\beta}) 
= \sum_{i \in I_<} \frac{X_i^\top(\beta^*-\hat{\beta})+\xi_i}{\sqrt{n}} \frac{X_i^\top}{\sqrt{n}} (\beta^*-\hat{\beta}) \\
&=  T_{a}+T_{b},
\end{align*}
where
$$
T_{a}={\left\|\frac{X_{I<}}{\sqrt{n}} (\beta^*-\hat{\beta})\right\|_2^2}, \qquad T_{b}=\frac{1}{\sqrt{n}} {\xi_{I<}^\top} \frac{X_{I_<}}{\sqrt{n}}(\beta^*-\hat{\beta}).
$$
First, we evaluate $T_2$ and $T_3$, because they can be easily evaluated. Next, we evaluate $T_a$ and $T_b$. 

\begin{lemma}
	\label{l:T2}
	Assume the conditions used in Proposition~\ref{p:coe-1-2-norm}. Then, with a high probability, we have
	\begin{align*}
	\left| T_2\right|
	\leq C_2 \|\Sigma^{\frac{1}{2}} (\beta^*-\hat{\beta})\|_2 , \qquad C_2=\lambda_o \sqrt{C_{cut}}\, g(C_{cut}).
	\end{align*}
\end{lemma}

\begin{proof}
	Let $u$ be the $n$-dimensional vector whose $i$th component is ${\rm sgn}(r_i(\hat{\beta}))$ for $i \in I_>$ and other components are zero. We see that $u$ is a $C_{cut}=\# I_>$ sparse vector and $\|u\|_2 \le \sqrt{C_{cut}}$. From Proposition \ref{p:DT19-4_2}, we have
	\begin{align*}
	\left| T_2 \right|
	&=
	\lambda_o \left| u^\top \frac{X}{\sqrt{n}} (\beta^*-\hat{\beta}) \right| \le \lambda_o  \|\Sigma^{\frac{1}{2}} (\beta^*-\hat{\beta})\|_2 \|u\|_2 g(C_{cut}) \\
	& \le 
	\lambda_o  \|\Sigma^{\frac{1}{2}} (\beta^*-\hat{\beta})\|_2\sqrt{C_{cut}}\,g(C_{cut}).
	\end{align*}
\end{proof}

\begin{lemma}
	\label{l:T3}
	Assume the conditions used in Proposition~\ref{p:coe-1-2-norm}. Then, with a high probability, we have
	\begin{align*}
	|T_3| 
	\leq  C_3 \|\Sigma^{\frac{1}{2}}(\beta^*-\hat{\beta})\|_2, \qquad C_3=\lambda_o \sqrt{o}\, g(o).
	\end{align*}
\end{lemma}
\begin{proof}
	Let $u$ be the $n$-dimensional vector whose $i$th component is $\psi(r_i(\hat{\beta}))$ for $i \in I_o$ and other components are zero. We see that $u$ is an $o=\# I_o$ sparse vector and $\|u\|_2 \le \sqrt{o}$ since $|\psi(u)| \le 1$. 
	From Proposition \ref{p:DT19-4_2}, we have
	\begin{align*}
	|T_3|
	&= 	\lambda_o \left|u^\top \frac{X}{\sqrt{n}} (\beta^*-\hat{\beta}) \right|
	\leq  \lambda_o \|\Sigma^{\frac{1}{2}} (\beta^*-\hat{\beta})\|_2 \|u\|_2 g(o) \\
	&\leq \|\Sigma^{\frac{1}{2}} (\beta^*-\hat{\beta})\|_2 \lambda_o \sqrt{o} g(o).
	\end{align*}
\end{proof}

\begin{lemma}
	\label{l:T11}
	Assume  the conditions used in Proposition~\ref{p:coe-1-2-norm}. Then, with a high probability, we have
	\begin{align*}
	\sqrt{T_a}
	\geq (a_1-C_{a1}-C_{a2}) \|\Sigma^{\frac{1}{2}}(\beta^*-\hat{\beta})\|_2, 
	\end{align*}
	where
	\begin{align*}
	C_{a1}=  1.2c_\kappa \sqrt{ \frac{2\rho^2s \log d}{n}}, \quad C_{a2}= g(C_{cut}+o).
	\end{align*}
\end{lemma}
\begin{proof}
	We see
	\begin{align*}
	T_a
	=  \left\| \frac{X_{I_<}}{\sqrt{n}}(\beta^*-\hat{\beta})\right\|_2^2 =  \left\|\frac{X}{\sqrt{n}}(\beta^*-\hat{\beta}) \right\|_2^2- 
	\left\|\frac{X_{I_> \cup I_o}}{\sqrt{n}}(\beta^*-\hat{\beta})\right\|_2^2 \geq 0
	\end{align*}
	and
	\begin{align*}
	\sqrt{T_a}
	&=  \sqrt{ \left\|\frac{X}{\sqrt{n}}(\beta^*-\hat{\beta}) \right\|_2^2- 
	\left\|\frac{X_{I_> \cup I_o}}{\sqrt{n}}(\beta^*-\hat{\beta})\right\|_2^2} \\
	& \geq \left\|\frac{X}{\sqrt{n}}(\beta^*-\hat{\beta}) \right\|_2- 
	\left\|\frac{X_{I_> \cup I_o}}{\sqrt{n}}(\beta^*-\hat{\beta})\right\|_2,
	\end{align*}
	because $\sqrt{A^2-B^2} \geq A-B$ for $A \geq B \geq 0$. From Proposition \ref{p:DT19-3_2}, we have
	\begin{align*}
	\left\|\frac{X}{\sqrt{n}}(\beta^*-\hat{\beta})\right\|_2 \geq  C_\kappa\|\Sigma^{\frac{1}{2}}(\beta^*-\hat{\beta})\|_2,
	\end{align*}
	where $C_\kappa=a_1- C_{a1}$. 
	Let $u$ be the $n$-dimensional vector whose $i$th component is $\frac{X_{i}}{\sqrt{n}} (\beta^*-\hat{\beta})$ for $i \in I_> \cup I_o$ and other components are zero. This is a $C_{cut}+o=\# I_>+\# I_o$ sparse vector. 
	From Proposition \ref{p:DT19-4_2}, we see
	\begin{align}
	\label{i:sample-cov-C}
	\|u\|_2^2=\left\|\frac{X_{I_> \cup I_o}}{\sqrt{n}}(\beta^*-\hat{\beta})\right\|_2^2  \leq \|\Sigma^{\frac{1}{2}}(\beta^*-\hat{\beta})\|_2 \|u\|_2 g(C_{cut}+o)
	\end{align}
	and then we have $\|u\|_2 \leq \|\Sigma^{\frac{1}{2}}(\beta^*-\hat{\beta})\|_2g(C_{cut}+o)$. Combining the results, the proof is complete.
\end{proof}

\begin{lemma}
	\label{l:Tb}
	Assume the conditions used in Proposition~\ref{p:coe-1-2-norm}. Then, with a high probability, we have
	\begin{align*}
	&|T_b| \leq C_b \|\Sigma^{\frac{1}{2}}(\beta^*-\hat{\beta})\|_2,
	\end{align*}
	where
	$$ {C_b}=2 c_\kappa \sqrt{\frac{2\sigma^2\rho^2s\log(d/\delta)}{n}}+ g(C_{cut}+o)  \sqrt{\frac{2 \sigma^2 (C_{cut} + o) \log (n/\delta)}{n}}.$$
\end{lemma}

\begin{proof}
	We see
	\begin{align*}
	|T_b|
	&=\left |  {\xi_{I_<}^\top} \frac{X_{I_<}}{{n}} (\beta^*-\hat{\beta})\right| \\
	&={\frac{1}{\sqrt{n}}}\left |  {\xi^\top} \frac{X}{\sqrt{n}} (\beta^*-\hat{\beta}) - {\xi^\top_{I_> \cup I_o}} \frac{X_{I_> \cup I_o}}{\sqrt{n}} (\beta^*-\hat{\beta}) \right| \\
	&\le \left |  {\xi^\top} \frac{X}{{n}} (\beta^*-\hat{\beta}) \right| +{\frac{1}{\sqrt{n}}} \left| {\xi^\top_{I_> \cup I_o}} \frac{X_{I_> \cup I_o}}{\sqrt{n}} (\beta^*-\hat{\beta}) \right|.
	\end{align*}
	From Propositions~\ref{p:DalTho2019Outlierlemma2}~and~\ref{p:coe-1-2-norm},
	\begin{align*} 
	\left|{\xi^\top} \frac{X}{{n}} (\beta^*-\hat{\beta})\right|
	&\le \| \beta^*-\hat{\beta} \|_1 \left\| {\xi^\top} \frac{X}{{n}} \right\|_\infty  
	\le
	\|\beta^*-\hat{\beta}\|_1 2  \sqrt{\frac{2\sigma^2\rho^2\log(3d/\delta)}{n}}\\
	&\le\|\Sigma^{\frac{1}{2}} (\beta^*-\hat{\beta})\|_2 2 c_\kappa \sqrt{\frac{2\sigma^2\rho^2s\log(3d/\delta)}{n}}.
	\end{align*}
	Let $u$ be the $n$-dimensional vector whose $i$th component is $\xi_{i}$ for $i \in I_> \cup I_o$ and other components are zero. This is a $C_{cut}+o=\# I_>+\# I_o$ sparse vector. 
	From Proposition~\ref{p:DT19-4_2}, 
	\begin{align*}
	\left| {\xi^\top_{I_> \cup I_o}} \frac{X_{I_> \cup I_o}}{\sqrt{n}} (\beta^*-\hat{\beta}) \right| =  \left| u^\top \frac{X}{\sqrt{n}} (\beta^*-\hat{\beta}) \right| \le \|\Sigma^{\frac{1}{2}} (\beta^*-\hat{\beta})\|_2\| u \|_2 g(C_{cut}+o).
	\end{align*}
	From Proposition~\ref{p:DalTho2019Outlierlemma2}, 
	\begin{align*}
	\|u\|_2^2 \le (C_{cut}+o) \|\xi\|_\infty^2 \le 2 \sigma^2 (C_{cut}+o) \log (n/\delta).
	\end{align*}
	Combining the above results, the proof is complete. 
\end{proof}

Combining the above results, we can easily show the following proposition.
\begin{proposition}
	\label{p:no-noi}
	Assume the conditions used in Proposition~\ref{p:coe-1-2-norm}. Suppose $a_1>0$. Then, with a high probability, we have
	\begin{align*}
	\frac{a_1^2}{2} \|\Sigma (\beta^*-\hat{\beta})\|_2  \leq 2(C_{a1}^2+C_{a2}^2) \|\Sigma^{\frac{1}{2}}(\beta^*-\hat{\beta})\|_2 + C_b+C_2+C_3+\lambda_s c_\kappa \sqrt{s}.
	\end{align*}
\end{proposition}
\begin{proof} 
	We have 
	\begin{align*}
	\lambda_o \sum_{i=1}^n \psi \left( r_i(\hat{\beta}) \right) \frac{X_i^\top}{\sqrt{n}} (\beta^*-\hat{\beta}) 
	&= T_1+T_2+T_3 \geq T_a -|T_b|-|T_2|-|T_3|.
	\end{align*}
	Let $\betamse=\|\Sigma^{\frac{1}{2}}(\beta^*-\hat{\beta})\|_2 $. 
	From Proposition~\ref{p:coe-1-2-norm}, we have {$\|\beta^*-\hat{\beta}\|_1 \leq c_\kappa \sqrt{s}\betamse$}. Combining these two inequalities on \eqref{eq:basic}, we have 
	\begin{align*}
	T_a  \leq  |T_b|+|T_2|+|T_3|+ \lambda_s c_\kappa \sqrt{s} \betamse.
	\end{align*}
	From $T_a \geq 0$, we can take a square root on both sides. From Lemmas~\ref{l:T2},~\ref{l:T3}~and~\ref{l:Tb}, we see
	\begin{align*}
	\sqrt{T_a}  &\leq \sqrt{|T_b|+|T_2|+|T_3|+ \lambda_s c_\kappa \sqrt{s} \betamse }\\
	& \leq \sqrt{(C_b+C_2+C_3+ \lambda_s c_\kappa \sqrt{s}) \betamse }.
	\end{align*}
	From Lemma~\ref{l:T11}, we have $\sqrt{T_a} \ge (a_1-C_{a1}-C_{a2})E$. Then, 
	\begin{align*}
	a_1 \betamse  &\leq (C_{a1}+C_{a2}) E+ \sqrt{(C_b+C_2+C_3+ \lambda_s c_\kappa \sqrt{s}) \betamse }
	\end{align*}
	and
	\begin{align*}
	a_1^2 \betamse  & \leq \left( (C_{a1}+C_{a2}) \sqrt{\betamse}+ \sqrt{C_b+C_2+C_3+ \lambda_s c_\kappa \sqrt{s} } \right)^2 \\
	& \le 2 \left( (C_{a1}+C_{a2})^2 \betamse+ C_b+C_2+C_3+ \lambda_s c_\kappa \sqrt{s} \right).
	\end{align*}
	The proof is completed from $(C_{a1}+C_{a2})^2 \le 2(C_{a1}^2+C_{a2}^2)$.
\end{proof}

\section{Proof of main theorem}
\label{sec:put-together}

\subsection{Notation}

Let
$$r_1 = \sqrt{\frac{s\log d}{n}}, \qquad r_2 = \frac{o}{n}  \sqrt{\log \frac{n}{o} \log n}, $$
\begin{align*}
r_{21}= \sqrt{\frac{o}{n} \log \frac{n}{o}}, \qquad
r_{22}= \sqrt{\frac{o}{n}\log n } \quad (\ge r_{21}).
\end{align*}
Then,
\begin{align*}
r_{n,d,s,o}=r_1+r_2, \qquad r_2=r_{21}r_{22}. 
\end{align*}

Let
\begin{align*}
\eta_\delta= \sqrt{\frac{\log (n/\delta)}{ \log n}}, \qquad 
\eta_4= \sqrt{\frac{4+\log(n/o)}{\log (n/o)}}.
\end{align*}
These are larger than 1 and bounded above by some constants, as shown in Lemma~\ref{l:eta_const}. 

Let $\bar{A}$ be an upper bound of $A$. In this section, we consider two cases; $C_{cut} \le o$ and $C_{cut}>o$. The corresponding upper bounds are denoted by $\bar{A}^<$ and $\bar{A}^>$.

\subsection{Theorem}

We present a general estimation error bound in Theorem~\ref{t:main}. By selecting a special tuning parameter in Theorem~\ref{t:main}, we can have the main theorem, which is given in Theorem~\ref{t:main2}. In this section, we prove Theorem~\ref{t:main} and the main theorem. 
 
\begin{theorem}
	\label{t:main}
	Consider the optimization problem (\ref{o:robust-lasso}).
	Suppose that $\Sigma$ satisfies $\mr{RE}(s,5,\kappa)$.
	Assume that $\delta$ and $n$ satisfy
\begin{align*}
{\rm (c1)} & \  \ \delta \in (0,1/7], \, n \geq 100 , \\
{\rm (c2)} & \ \ \sqrt{\frac{\log(d/\delta)}{ n}} \leq \sqrt{3} -\sqrt{2} \quad (\text{This implies } 2\log (d/\delta) \leq n \text{ from } n \geq 100.), \\
{\rm (c3)} & \ 2\sqrt{n \log (1/\delta)} + 2\log(1/\delta) \leq n, \\
{\rm (c4)} & \  \  a_1>3/4, \\
{\rm (c5)} & \  \  b_1<1/4,
\end{align*}
where 
\[ a_1= 1-\frac{4.3+\sqrt{2 \log(9/\delta)}}{\sqrt{n}}, \quad b_1 = \sqrt{\frac{2}{n}}\left(4.8+\sqrt{\log\frac{81}{\delta}}\right). \]
	Suppose that $\lambda_s$ and $\lambda_o$ satisfy
	\begin{align}
	\label{ine:par}
	& \lambda_o\geq C_{\lambda_o}\sqrt{\frac{2\sigma^2\log (n/\delta)}{n}},\quad
	\lambda_s \geq \frac{4\sqrt{2}}{\sqrt{3}} C_{\lambda_s} \lambda_o, \\
	\label{cond0}
	& 8 \max \left(3.6\sqrt{\frac{2\rho^2 \log d}{n}}, 2.4\frac{\lambda_s}{\lambda_o}\sqrt{\frac{2\log n}{n}}\right)\sqrt{  \frac{s}{\kappa^2} + \frac{6.25 o \lambda_o^2}{\lambda_s^2}}  \leq C_{n,\delta},
	\end{align}
	where 
	\begin{align}
	\label{cond:Clambdas}
	C_{\lambda_s} &= C_z + \sqrt{2\frac{o}{s}}g(o) ,\\
	C_{n,\delta} &= \sqrt{a_1^2 + b_1 + 1/4} -\sqrt{2(b_1+1/4)}.
	\end{align}
	Assume that \eqref{i:so} is satisfied with $c_0=5$, $C_{\lambda_o}$ is a sufficiently large constant such that $C_{\lambda_o}\ge 2$ and
	\begin{align*}
	C_>=\frac{9}{32} - 2\times 9.6^2 e \bar{\eta}_4 \frac{ C_{\lambda_o} }{ C_{\lambda_o}^2-1 }>0,
	\end{align*}
	where $\bar{\eta}_4$ is a constant given in Lemma~\ref{l:eta_const}. 
	Then, with probability at least $1-7\delta$, the optimal solution $\hat{\beta}$ satisfies
	\begin{align}
	\|\Sigma^{\frac{1}{2}}(\beta^*-\hat{\beta})\|_2 & \leq \frac{32}{9}\left(R+2 \bar{C}_{02}^< \nu_E+ \bar{C}_{b2}^<+\bar{C}_{2}^{<} \right), & \mbox{if $C_{cut} \le o$}, \label{eq1:t:main} \\
	 & \leq \frac{1}{C_>} \left( R  +  2 \bar{C}_{02}^> \nu_E+ \bar{C}_{b21}^{>} + \bar{C}_{21}^{>} \right) , & \mbox{if $C_{cut} > o$}, \label{eq2:t:main}
	\end{align}
where $\bar{C}$s and $\nu_E$ are given in Lemmas~\ref{lemma:cfirst:fuji}, \ref{lemma:C3+}, \ref{lemma:Cetc+<}, \ref{lemma:Ccutupper2}, Proposition~\ref{p:DT19-1} and \eqref{e:Cupper5},
\[  R= 2 \bar{C}_{01} \nu_E  +  {\bar{C}_{b1}}  +{\bar{C}_3} + \lambda_s c_\kappa \sqrt{s}.  \]
\end{theorem}

\begin{theorem}
	\label{t:main2}
	Consider the optimization problem (\ref{o:robust-lasso}).
	Assume the same conditions as in Theorem~\ref{t:main} except for 

\begin{eqnarray}
\lambda_o =  C_{\lambda_o}\sqrt{\frac{2\sigma^2\log (n/\delta)}{n}}, \quad \lambda_s = \frac{4\sqrt{2}}{\sqrt{3}} C_{\lambda_s} \lambda_o.
\label{e:par}
\end{eqnarray}
The, with probability at least $1-7\delta$, the optimal solution $\hat{\beta}$ satisfies
	\begin{align*}
	\|\Sigma^{\frac{1}{2}}(\beta^*-\hat{\beta})\|_2 \le C_{\delta,\kappa,\rho,\sigma} r_{n,d,s,o},
	\end{align*}
	where $C_{\delta,\kappa,\rho,\sigma}$ is a constant depending on $\delta, \kappa,\rho,\sigma$.
\end{theorem}


\subsection{Outline of The Proof of Theorem~\ref{t:main}} \label{sec:Outline_Proof_main}

We restrict the sample space with probability at least $1-\delta$ via Propositions~\ref{p:DT19-3}, \ref{p:DT19-4}, \ref{p:DalTho2019Outlierlemma2}, \ref{p:XxiBernstein} and \ref{p:LauMas2000Adaptivelemma1}. Hence, the theorem hold with probability at least $1-7\delta$.

We prove Theorem~\ref{t:main}, using the basic inequality given in Proposition~\ref{p:no-noi}. Here, we again write the basic inequality:
\begin{eqnarray}
\frac{a_1^2}{2} \betamse &\leq 2(C_{01}+C_{02}) \betamse  +  C_{b}+C_2+C_3 +\lambda_s c_\kappa \sqrt{s},
\label{e:CentralInequality:fuji}
\end{eqnarray}
where
\begin{align*}
& C_{01} =C_{a1}^2=(1.2c_\kappa)^2 \frac{2\rho^2s \log d}{n}, \quad
C_{02} = C_{a2}^2=g(C_{cut}+o) ^2, \\
C_b &=C_{b1}+C_{b2}, \quad 
C_{b1}=2c_\kappa \sqrt{\frac{2\sigma^2\rho^2 s\log(d/\delta)}{n}}, \\ 
& C_{b2} =g(C_{cut} + o) \sqrt{ 2\sigma^2 (C_ {cut}/o + 1 ) } C_{b2}', \quad 
C_{b2}' =\sqrt{ \frac{o\log (n/\delta)}{n} }, \\
C_2& =\lambda_o \sqrt{C_{cut}} g(C_{cut}), \quad 
C_3 =\lambda_o\sqrt{o}g(o). 
\end{align*}
The basic inequality holds when we assume the conditions used in Proposition~\ref{p:coe-1-2-norm}; (c1) and (c2), the restricted eigenvalue condition, $\lambda_s-C_{\lambda_s}\lambda_o>0$ and \eqref{i:so}. The same conditions are assumed in Theorem~\ref{t:main} except for $\lambda_s-C_{\lambda_s}\lambda_o>0$. This condition holds from \eqref{ine:par}.

We will evaluate each term of the basic inequality by the error orders $r_1$, $r_2$, $r_{21}$ and $r_{22}$. Combining the results, we will prove Theorem~\ref{t:main}. Through Section~\ref{sec:put-together}, we assume the conditions used in Theorem~\ref{t:main}.

Some terms are immediately evaluated as follows.

\begin{lemma} 
\label{lemma:cfirst:fuji}
We have
\begin{align*}
C_{01}& =(1.2c_\kappa)^2 \frac{2\rho^2s \log d}{n} \le (1.2c_\kappa)^2 {2\rho^2} r_1^2=:\bar{C}_{01}, \\
C_{b1}&=2c_\kappa \sqrt{\frac{2\sigma^2\rho^2 s\log(d/\delta)}{n}} \le 2c_\kappa \sqrt{2\sigma^2\rho^2} \sqrt{1+\log(1/\delta)}r_1 =: \bar{C}_{b1}, \\
C_{b2}'&=\sqrt{ \frac{o \log (n/\delta)}{n} } = \eta_\delta r_{22}. 
\end{align*}
\end{lemma}
\begin{proof} 
The first and third inequalities hold immediately from the definitions of $r_1$, $r_{22}$ and $\eta_\delta$. The second inequality follow from $\log (d/\delta) = \log(d)+\log(1/\delta) \leq \log(d)+ \log(1/\delta) \log (d)$.
\end{proof}

The remaining terms are related to the function $g(m)$ and $C_{cut}$. First, we show some simple properties of $g(m)$. Next, we consider $C_{cut}$ with two cases; $C_{cut}\le o$ and $C_{cut}>o$. The case $C_{cut}\le o$ is easily treated. The case $C_{cut}> o$ is treated later in detail. 


\subsection{Constant Bounds}

In this subsection, we give some constant bounds. 

\begin{lemma}
	\label{l:Cndelta}
	We have
	\begin{align}
	\label{eq:Cndelta1}
	\underline{C} \leq C_{n,\delta} \leq \overline{C}, 
	\end{align}
where $\underline{C}=\sqrt{17/16}-1>0$ and $\overline{C}=(\sqrt{5}-\sqrt{2})/{2}$.
\end{lemma}
\begin{proof}
	We can easily see that $C_{n,\delta} = \sqrt{a_1^2 + b_1 + 1/4} -\sqrt{2(b_1+1/4)}$ is monotonically increasing on $a_1$ and decreasing on $b_1$. Since ${ 3/4} \le a_1 \le 1$ and $0 \le b_1 \le 1/4$, we have
	\begin{align*}
	C_{n,\delta} &\leq \sqrt{1+1/4}-\sqrt{2(1/4)}=\overline{C}, \\
	C_{n,\delta} &\geq \sqrt{(3/4)^2+1/4+ 1/4}-\sqrt{2(1/4+1/4)} =\underline{C}.
	\end{align*}
\end{proof}

\begin{lemma} \label{l:r_const} 
We have
\begin{align*}
	r_1 \le \overline{C} \kappa/28.8\sqrt{2}\rho, \qquad r_{22} \le \overline{C}/ 19.2\sqrt{12.5}.
	\end{align*}
\end{lemma}

\begin{proof}
	From \eqref{cond0}, we know
	\begin{align*}
	C_{n,\delta} \geq 8 \max\left(3.6\sqrt{\frac{2\rho^2 \log {d}}{n}}, 2.4\frac{\lambda_s}{\lambda_o}\sqrt{\frac{2\log n}{n}}\right) \sqrt{  \frac{s}{\kappa^2} + \frac{6.25 o \lambda_o^2}{\lambda_s^2}}.
	\end{align*}
	From this inequality, we see
	\begin{align*}
	C_{n,\delta}
	& \geq  8\times 3.6\sqrt{\frac{2\rho^2 \log {d}}{n}} \sqrt{  \frac{s}{\kappa^2}} \geq \frac{28.8\sqrt{2}\rho}{\kappa} \sqrt{\frac{ s\log {d}}{n}}, \\
	C_{n,\delta}
	& \geq  8\times 2.4\frac{\lambda_s}{\lambda_o}\sqrt{\frac{2\log n}{n}}\sqrt{ \frac{6.25 o \lambda_o^2}{\lambda_s^2}}   \geq 19.2\sqrt{12.5} \sqrt{\frac{o\log n}{n}}.
	\end{align*}
	The proof is complete from Lemma~\ref{eq:Cndelta1}.
\end{proof}

\begin{lemma} 
\label{l:eta_const}
We have
\begin{align}
\label{ine:delta}
\eta_\delta &= \sqrt{\frac{\log (n/\delta)}{ \log n}} \leq  \sqrt{\frac{\log 100/\delta}{\log 100}},\\
\label{ine:4}
\eta_4 &= \sqrt{\frac{4+\log(n/o)}{\log (n/o)}} \leq \sqrt{\frac{4+\log C_{on}}{\log C_{on}}}.
\end{align}
where $C_{on} = (19.2\sqrt{12.5})^2\log 100/\overline{C}$
\end{lemma}
\begin{proof}
	Note that $(a+\log(x))/\log(x)$ with $a>0$ and $x>0$ is a monotone decreasing function of $x$. 
	The first inequality holds from $n \geq 100$. Next, we consider the second inequality. From Lemma~\ref{l:eta_const} and $r_{22}=\sqrt{o \log n/n}$, we see
	\begin{align*}
	{o}/{n} \le (\overline{C}/ 19.2\sqrt{12.5})^2/\log n \le (\overline{C}/ 19.2\sqrt{12.5})^2/\log 100 =: 1/C_{on},
	\end{align*}
	which implies 
	\[ \eta_4= \sqrt{\frac{4+\log(n/o)}{\log (n/o)}} \le \sqrt{\frac{4+\log C_{on}}{\log C_{on}}} \]. 
\end{proof}

\begin{lemma} 
	\label{lemma:Cr_const}
We have 
\[ C_r = \frac{1}{1-{2\sigma^2 \log(n/\delta)}/{\lambda_o^2 n}} \leq \frac{C_{\lambda_o}^2}{C_{\lambda_o}^2-1}. \]
\end{lemma}
\begin{proof}
From the condition~\eqref{ine:par}, we know $\lambda_o \geq C_{\lambda_o}\sqrt{{2\sigma^2\log (n/\delta)}/{n}}$, which implies
\begin{align*}
\frac{1}{C_r} 
= 1- \frac{2\sigma^2\log(n/\delta)}{n\lambda_o^2} \ge 1-\frac{1}{C_{\lambda_o}^2}.
\end{align*}
The proof is complete.
\end{proof}

\subsection{Evaluation of $g(\cdot)$}

Let
$$ g(m)=g_1+g_2(m), $$
where
\begin{eqnarray*}
g_1 &=& \sqrt{\frac{2}{n}}\left(4.8+ \sqrt{\log \frac{81}{\delta}}\right) + 1.2c_\kappa \sqrt{\frac{2\rho^2s\log d}{n}}, \\
g_2(m) &=& 4.8\sqrt{e}\sqrt{\frac{m}{n}}\sqrt{4+\log\frac{n}{m}}.
\end{eqnarray*}
We can easily see that $g_2(m)$ and $g(m)$ are monotone increasing functions of $m$. 

\begin{lemma}  
	\label{l:g1r}
	We have
	\begin{align*}
	g_1  \le c_g r_1,
	\end{align*}
	where $c_g=4.8\sqrt{2}+\sqrt{2\log (81/\delta)} +1.2c_\kappa \sqrt{2\rho^2}$. 
\end{lemma}
\begin{proof} 
	We have
	\begin{align*}
	\sqrt{\frac{2}{n}}\left(4.8+ \sqrt{\log \frac{81}{\delta}}\right) &\leq  \sqrt{2}\left(4.8+\sqrt{\log \frac{81}{\delta}}\right)r_1.
	\end{align*}
	and
	\begin{align*}
	1.2c_\kappa \sqrt{\frac{2\rho^2s\log d}{n}} \le 1.2c_\kappa \sqrt{2\rho^2}r_1.
	\end{align*}
\end{proof}

\begin{lemma}  
	\label{l:lsog2:fuji}
	We have
	\begin{align*}
	g_2(o) = 4.8\sqrt{e}\eta_4 r_{21}.
	\end{align*}
\end{lemma}
\begin{proof} 
	We have
	\begin{align*}
	&g_2(o)
	= 4.8\sqrt{e}\sqrt{\frac{o}{n}}\sqrt{4+\log\frac{n}{o}}= 4.8\sqrt{e} \eta_4 \sqrt{\frac{o}{n}} \sqrt{\log \frac{n}{o}}.
	\end{align*}
\end{proof}

The following lemma holds immediately from Lemmas~\ref{l:g1r}~and~\ref{l:lsog2:fuji}.
\begin{lemma} \label{lemma:C3+} 
	We have 
	\begin{align*}
	C_3 =\lambda_o \sqrt{o}\,g(o) \le (c_g r_1  + 4.8\sqrt{e}\eta_4 r_{21} )  \sqrt{o} \lambda_o =:\bar{C}_3.
	\end{align*}
\end{lemma}


\subsection{Simple estimation bound of $E=\|\Sigma^{\frac{1}{2}}(\beta^*-\hat{\beta})\|_2$}

First, we introduce the concept of augmented transfer principle.

\begin{definition}[Definition 1 of \cite{DalTho2019Outlier}]
	\label{d:DT19-1}
	We say that $X$ satisfies {the} {augmented transfer principle} $\mr{ATP}_{\Sigma}(c_1,c_2,c_3)$ for some positive numbers $c_1$, $c_2$ and $c_3$, when for any $v \in \mbb{R}^d$ and $u \in \mbb{R}^n$, we have
	\begin{align*}
	\left\|\frac{X}{\sqrt{n}}v+u\right\|_2 \geq c_1 \left(\|\Sigma^{\frac{1}{2}} v\|_2 + \|u\|_2\right)-c_2\|v\|_1 -c_3\|u\|_1.
	\end{align*}
\end{definition}

The following lemma is a slight modification of Lemma 7 of \cite{DalTho2019Outlier}, because we suspect the correctness. The proof of the following lemma is given in Appendix~\ref{appendix:Lemma7DT}. 

\begin{lemma}[Modification of Lemma 7 of \cite{DalTho2019Outlier}]
	\label{l:DT19-7-2}
	Let $Z \in \mbb{R}^{n\times d}$ be a random matrix satisfying 
	\begin{align*}
	\left\|\frac{Z}{\sqrt{n}}v \right\|_2 \geq a_1 \|\Sigma^{\frac{1}{2}} v\|_2 -a_2\|v\|_1
	\end{align*}
	and
	\begin{align*}
	\left|u^\top\frac{Z}{\sqrt{n}}v\right| \leq b_1 \|\Sigma^{\frac{1}{2}} v\|_2\|u\|_2 + b_2\|v\|_1\|u\|_2 + b_3\|\Sigma^{\frac{1}{2}}v\|_2\|u\|_1
	\end{align*}
	for some positive constants $a_1 \in (0,1),\ a_2,\ b_1,\ b_2,\ b_3$. Then, for any $\alpha>0$, $Z$ satisfies
	\begin{align*}
	\left\|\frac{Z}{\sqrt{n}}v+u\right\|_2 \geq c_1 \left(\|\Sigma^{\frac{1}{2}}v\|_2 +\|u\|_2\right)-c_2\|v\|_1-c_3\|u\|_1
	\end{align*}
	with the constants $c_1=\sqrt{a_1^2+b_1 + \alpha^2}-\sqrt{2(b_1+\alpha^2)},\ c_2 = a_2+b_2/\alpha,\ c_3 = b_3/\alpha$. 
	 If $a_1^2 > b_1+ \alpha^2$, then we have $c_1>0$.
\end{lemma}

We can obtain a simple estimation bound of $E=\|\Sigma^{\frac{1}{2}}(\beta^*-\hat{\beta})\|_2$ from Proposition~1 of \cite{DalTho2019Outlier}. As seen later in Lemma~\ref{l:O_nuE}, this bound is roughly of order $r_1+r_{22}$
\begin{proposition} 
	\label{p:DT19-1}
We have
	\begin{align*}
	\|\Sigma^{\frac{1}{2}}(\beta^*-\hat{\beta})\|_2^2 + \|\theta^*-\hat{\theta}\|_2^2  \leq \nu_E^2,
	\end{align*}
where
\[
\nu_E=\frac{6}{C_{n,\delta}^2}\sqrt{\frac{\lambda_s^2 s}{\kappa^2} + 6.25 \lambda^2_o o}.
\]
In addition, we have
	\begin{align*}
	\|\Sigma^{\frac{1}{2}} (\beta^*-\hat{\beta})\|_2  \leq \nu_E.
	\end{align*}
\end{proposition}

\begin{proof}
The first result is the same as in Proposition~1 of \cite{DalTho2019Outlier}. It is enough to verify the conditions assumed in Proposition~1 of \cite{DalTho2019Outlier}. The same conditions are assumed in Theorem~\ref{t:main}, except for a similar condition to \eqref{cond0} and the conditions that 
$\lambda_o \geq (2/\sqrt{n})\|\xi\|_\infty$, $\lambda_s \geq (2/n) \|X^\top \xi\|_\infty$, and 
$X$ satisfies $\mr{ATP}_\Sigma(c_1;c_2;c_3)$ with the constants $c_1 = C_{n,\delta}>0$, $c_2 = 3.6\sqrt{{2\rho^2 \log d}/{n}}$, $c_3 = 2.4\sqrt{{2 \log n}/{n}}$.
The condition \eqref{cond0} is just a sight modification of the condition assumed in Theorem~\ref{t:main}. 
As a result, the proof is complete by verifying that these three conditions hold from the conditions assumed in Theorem~\ref{t:main}. 
The ATP condition is proved in Proposition~\ref{p:ATP}. 
The condition $\lambda_o \geq (2/\sqrt{n})\|\xi\|_\infty$ can be easily proved from \eqref{ine:par} and Proposition~\ref{p:DalTho2019Outlierlemma2}. 
The condition $\lambda_s \geq (2/n) \|X^\top \xi\|_\infty$ is proved as follows. From \eqref{ine:par}, 
\begin{align*}
\lambda_s &\ge \frac{4\sqrt{2}}{\sqrt{3}}C_{\lambda_s}\lambda_o \ge \frac{4\sqrt{2}}{\sqrt{3}} C_z \lambda_o = \frac{4\sqrt{2}}{\sqrt{3}} \sqrt{ 3 \frac{\rho^2\sigma^2}{n\lambda_o^2} \log\frac{d}{\delta} }\,\lambda_o 
 = 4\sqrt{2} \sqrt{ \frac{\rho^2\sigma^2}{n} \log\frac{d}{\delta} }  \ge \frac{2}{n} \|X^\top \xi\|_\infty,
\end{align*}
since the last inequality holds from Proposition~\ref{p:DalTho2019Outlierlemma2} and (c2). 
\end{proof}

\begin{proposition}
	\label{p:ATP}
$X$ satisfies $\mr{ATP}_\Sigma(c_1;c_2;c_3)$ with the constants 
	$c_1 = C_{n,\delta}$, $c_2 = 3.6\sqrt{{2\rho^2 \log d}/{n}}$, $c_3 = 2.4\sqrt{{2 \log n}/{n}}$.
\end{proposition}
\begin{proof} 
From Lemma~\ref{l:gau-wid-d} and Proposition~\ref{p:DT19-3}, we have 
	\begin{align*}
	\left\|\frac{Z}{\sqrt{n}}v \right\|_2 &\geq a_1 \|\Sigma^{\frac{1}{2}} v\|_2 -a_2\|v\|_1,
	\end{align*}
where
\begin{align*}
a_1= 1-\frac{4.3+\sqrt{2 \log(9/\delta)}}{\sqrt{n}}, \quad a_2 = 1.2\sqrt{\frac{2\rho^2 \log d}{n}}.
\end{align*}
From Lemma~\ref{l:gau-wid} and Corollary~\ref{p:DT19-4}, we have 
	\begin{align*}
	\left|u^\top \frac{Z}{\sqrt{n}}v\right| &\leq b_1 \|\Sigma^{\frac{1}{2}} v\|_2\|u\|_2 +b_2\|v\|_1\|u\|_2
	+b_3 \|\Sigma^{\frac{1}{2}} v\|_2 \|u\|_1,
	\end{align*}
where
\begin{align*}
b_1=\sqrt{\frac{2}{n}}\left(4.8+\sqrt{\log\frac{81}{\delta}}\right), \quad b_2= 1.2\sqrt{\frac{2\rho^2 \log d}{n}}, \quad b_3 = 1.2 \sqrt{\frac{2\log n}{n}}.
\end{align*}
Let
	\begin{align*}
	c_1 & = C_{n,\delta} = \sqrt{a_1^2+b_1+1/4}-\sqrt{2(b_1+1/4)}, \\
	c_2 & = a_2 +2 b_2  = 3.6 \sqrt{\frac{2\rho^2 \log d}{n}}, \quad c_3 = 2b_3 = 2.4 \sqrt{\frac{2\log n}{n}}.
	\end{align*}
	The condition $c_1>0$ holds from Lemma~\ref{l:Cndelta}. From Lemma~\ref{l:DT19-7-2} with $\alpha = 1/2$, $X$ satisfies $\mr{ATP}_\Sigma(c_1;c_2;c_3)$. 
\end{proof}

\subsection{Case of $C_{cut} \le o$}

\begin{lemma} 
	\label{lemma:g2cut:fuji}
	We have
	\begin{align*}
	g(C_{cut}) &\le g(o) = g_1+g_2(o) \le c_gr_1+4.8\sqrt{e}\eta_4 r_{21}, \\
	g(C_{cut}+o) &\le g(2o) = g_1+g_2(2o) \le c_gr_1+4.8\sqrt{2e}\eta_4 r_{21}.
	\end{align*}
\end{lemma}
\begin{proof} 
	From Lemmas~\ref{l:g1r} and \ref{l:lsog2:fuji}, we have $g_1 \leq g_gr_1$, $g_2(o) = 4.8 \sqrt{e}\eta_4 r_{21}$ and
	\begin{align*}
	g_2(2o) &
	= 4.8\sqrt{e}\sqrt{\frac{2o}{n}}\sqrt{4+\log\frac{n}{2o}}
	\le 4.8\sqrt{2e}\sqrt{\frac{o}{n}}\sqrt{4+\log\frac{n}{o}} \\
	&= 4.8\sqrt{2e} \eta_4 \sqrt{\frac{o}{n}} \sqrt{\log \frac{n}{o}}.
	\end{align*}
\end{proof}

\begin{lemma} \label{lemma:Cetc+<}
We have
\begin{align*}
C_{02} & \le 2 (c_g)^2 r_1^2 + 2\times 4.8^2 e\eta_4^2r_2 =:\bar{C}_{02}^{<}, \\
C_{b2}& \le
2 c_g \sqrt{ \sigma^2 } \eta_\delta r_1 r_{22} + 9.6 \sqrt{2e\sigma^2} \eta_4 \eta_\delta r_{2} =:\bar{C}_{b2}^{<}, \\
C_2&  \le \lambda_o \sqrt{o} (c_gr_1+4.8\sqrt{e}\eta_4 r_{21}) =:\bar{C}_{2}^{<}.
\end{align*}
\end{lemma}

\begin{proof}
From Lemmas~\ref{lemma:g2cut:fuji}~and~\ref{lemma:cfirst:fuji}, we see
\begin{align*}
C_{02} &= g(C_{cut}+o) ^2 \le g(2o)^2 \le (c_gr_1+4.8\sqrt{2e}\eta_4 r_{21})^2 \\
& \le 2\left\{ (c_g)^2 r_1^2 + 2\times 4.8^2 e\eta_4^2r_{21}^2 \right\} \le 2 (c_g)^2 r_1^2 + 2\times 4.8^2 e\eta_4^2r_2, \\
C_{b2}&=g(C_{cut} + o) \sqrt{ 2\sigma^2 (C_ {cut}/o + 1 )} C_{b2}' \le g(2o)\sqrt{ 4\sigma^2} C_{b2}' \\
 &\le ( c_gr_1+4.8\sqrt{2e}\eta_4 r_{21} ) \sqrt{ 4\sigma^2 } \eta_\delta r_{22} 
 = 2 c_g \sqrt{ \sigma^2 } \eta_\delta r_1 r_{22} + 9.6 \sqrt{2e\sigma^2} \eta_4 \eta_\delta r_{2}, \\
C_2& =\lambda_o \sqrt{C_{cut}} g(C_{cut}) \le \lambda_o \sqrt{o} g(o) \le \lambda_o \sqrt{o} (c_gr_1+4.8\sqrt{e}\eta_4 r_{21}) .
\end{align*}
\end{proof}

Using the basic inequality \eqref{e:CentralInequality:fuji} with 
the upper bounds obtained in Lemmas~\ref{lemma:cfirst:fuji}, \ref{lemma:C3+} and \ref{lemma:Cetc+<} and Proposition~\ref{p:DT19-1}, we can easily obtain the following proposition, which shows the estimation error \eqref{eq1:t:main} of $\betamse$ from (c4). 

\begin{proposition} \label{p:main:<=o}
We have
\begin{align*}
\frac{a_1^2}{2} \betamse &\leq 2({\bar{C}_{01}}+\bar{C}_{02}^<)\nu_E+ ({\bar{C}_{b1}}+\bar{C}_{b2}^<)+\bar{C}_{2}^{<}+{\bar{C}_3} +\lambda_s c_\kappa \sqrt{s},
\end{align*}
where $\bar{C}$s and $\nu_E$ are given in Lemmas~\ref{lemma:cfirst:fuji}, \ref{lemma:C3+} and \ref{lemma:Cetc+<} and Proposition~\ref{p:DT19-1}. 
\end{proposition}


\subsection{Case of $C_{cut} > o$}

We can obtain an upper bound of $C_{cut}$ from Proposition~\ref{p:C/n}. 

\begin{lemma}  
	\label{l:Ccut_upper:fuji}
	We have
	\begin{align*}
	C_{cut} &\leq  \upsilon^{cut}= \upsilon^{cut}_1 +\upsilon_2^{cut},
	\end{align*}
	where
	\begin{align*}
& \upsilon^{cut}_1 = 2C_r \left( \frac{\sqrt{2\sigma^2}}{\lambda_o^2} c_g r_1 + \frac{\sqrt{o}}{\lambda_o} c_g r_1 + \frac{\sqrt{o}}{\lambda_o} 4.8\sqrt{e} \eta_4 r_{21} +\sqrt{s}c_\kappa\frac{\lambda_s}{\lambda_o^2} \right)\betamse, \\
& \upsilon_2^{cut} =  C_{v2} \frac{1}{\lambda_o^2} \betamse,  \quad C_{v2}= 19.2\sqrt{2e\sigma^2}C_{r}.
	\end{align*}
\end{lemma}
\begin{proof} 
We see $ g(n-o) \le g(n) = g_1+g_2(n) \le c_gr_1+9.6\sqrt{e}$ and $g(o) \le c_gr_1 + 4.8\sqrt{e} \eta_4 r_{21}$ from Lemmas~\ref{l:g1r}~and~\ref{l:lsog2:fuji}. 
From Proposition~\ref{p:C/n}, we have
\begin{align*}
\frac{C_{cut}}{2C_r\betamse}
 &\leq \frac{\sqrt{2\sigma^2}}{\lambda_o^2} g(n-o) + \frac{\sqrt{o}}{\lambda_o}g(o) +\sqrt{s}c_\kappa\frac{\lambda_s}{\lambda_o^2} \\
 &\leq \frac{\sqrt{2\sigma^2}}{\lambda_o^2} c_g r_1  + \frac{\sqrt{2\sigma^2}}{\lambda_o^2} 9.6\sqrt{e} + \frac{\sqrt{o}}{\lambda_o} c_g r_1 + \frac{\sqrt{o}}{\lambda_o} 4.8\sqrt{e} \eta_4 r_{21} +\sqrt{s}c_\kappa\frac{\lambda_s}{\lambda_o^2}.
\end{align*}
The proof is complete. 
\end{proof}

Roughly speaking, {$\lambda_o^2 \upsilon_2^{cut}/E =C_{v2}=O(1)$}, but {$\lambda_o^2 \upsilon^{cut}_1/E \asymp r_1+\sqrt{o}\lambda_o (r_1+r_{21}) + \sqrt{s}\lambda_s$}, which is of order $O(r_{n,d,s,o})$, as shown later. Taking into consideration the difference between these orders, we will evaluate various terms.

Using Lemma~\ref{l:Ccut_upper:fuji}, we evaluate each term of the basic inequality, in a similar manner to the above. 

\begin{lemma}
\label{lemma:Ccutupper2}
	We have
	\begin{align*}
	C_{02}&=g(C_{cut}+o)^2 \\
 &\le (c_g)^2 r_1^2 + 9.6 c_g \sqrt{e} \eta_4 r_1 r_{21}\sqrt{\frac{\upsilon^{cut}}{o}+1 }  
+ 4.8^2 e \eta_4^2 r_{21}^2 \left( \frac{\upsilon^{cut}}{o}+1 \right)=:\bar{C}_{02}^{>}, \\
	C_2 &={\lambda_o}\sqrt{C_{cut}}g(C_{cut}) \le {\lambda_o} \left( c_g r_1 \sqrt{\upsilon^{cut}} + 4.8 \sqrt{e} \eta_4 r_{21} \frac{\upsilon^{cut}}{\sqrt{o}} \right)=:\bar{C}_{2}^{>}, \\
	{C_{b2}} &={\sqrt{2\sigma^2}C_{b2}'} \sqrt{\frac{C_{cut}}{o}+1} \, g(C_{cut}+o) \\
 &\leq \sqrt{2\sigma^2} \eta_\delta r_{22} \left( c_g r_1 \sqrt{\frac{\upsilon^{cut}}{o}+1} + 4.8 \sqrt{e} \eta_4 r_{21} \left( \frac{\upsilon^{cut}}{o}+1 \right) \right)=:\bar{C}_{b2}^{>} .
\end{align*}
\end{lemma}
\begin{proof}
In this proof, we often use $g_1 \le c_g r_1$ from Lemma~\ref{l:g1r}. 
We see
\begin{align*}
\sqrt{C_{cut}} g(C_{cut})
 & = \sqrt{C_{cut}} \{g_1+g_2(C_{cut})\} \\
 & \le c_g r_1 \sqrt{C_{cut}} + 4.8 \sqrt{e} \frac{C_{cut}}{\sqrt{n}} \sqrt{4+\log \frac{n}{C_{cut}}} \\
 & \le c_g r_1 \sqrt{C_{cut}} + 4.8 \sqrt{e} \frac{C_{cut}}{\sqrt{n}} \sqrt{4+\log \frac{n}{o}} \\
 & = c_g r_1 \sqrt{C_{cut}} + 4.8 \sqrt{e} \frac{C_{cut}}{\sqrt{n}} \eta_4 \sqrt{\log \frac{n}{o}} \\
 & = c_g r_1 \sqrt{C_{cut}} + 4.8 \sqrt{e} \frac{C_{cut}}{\sqrt{o}} \eta_4 r_{21} \\
 & = c_g r_1 \sqrt{C_{cut}} + 4.8 \sqrt{e} \eta_4 r_{21} \frac{C_{cut}}{\sqrt{o}}.
\end{align*}
The final formula is a monotone increasing function of $C_{cut}$. 
We know $C_{cut} \le \upsilon^{cut}$ from Lemma~\ref{l:Ccut_upper:fuji}.
By replacing $C_{cut}$ by the upper bound $\upsilon^{cut}$, the second inequality of the lemma is proved. 
We see
\begin{align*}
g(C_{cut}+o) 
&\le  c_g r_1 + 4.8 \sqrt{e} \sqrt{\frac{C_{cut}+o}{n} }\sqrt{4+\log \frac{n}{C_{cut}+o}} \\
&\le c_g r_1 + 4.8 \sqrt{e} \sqrt{\frac{C_{cut}+o}{n} } \sqrt{4+\log \frac{n}{o}} \\
&= c_g r_1 + 4.8 \sqrt{e} \sqrt{\frac{C_{cut}+o}{n} }\eta_4 \sqrt{\log \frac{n}{o}} \\
&= c_g r_1 + 4.8 \sqrt{e} \eta_4 r_{21} \sqrt{\frac{C_{cut}}{o}+1 }. 
\end{align*}
Hence,
\begin{align*}
g(C_{cut}+o)^2 
&\le (c_g)^2 r_1^2 + 9.6 c_g \sqrt{e} \eta_4 r_1 r_{21}\sqrt{\frac{C_{cut}}{o}+1 }  
+ 4.8^2 e \eta_4^2 r_{21}^2 \left( \frac{C_{cut}}{o}+1 \right), 
\end{align*}
\begin{align*}
\sqrt{\frac{C_{cut}}{o}+1}\, g(C_{cut}+o) 
&\le c_g r_1 \sqrt{\frac{C_{cut}}{o}+1} + 4.8 \sqrt{e} \eta_4 r_{21} \left( \frac{C_{cut}}{o}+1 \right) .
\end{align*}
Two final formulas are monotone increasing functions of $C_{cut}$. 
We know $C_{cut} \le \upsilon^{cut}$ from Lemma~\ref{l:Ccut_upper:fuji}.
By replacing $C_{cut}$ by the upper bound $\upsilon^{cut}$, 
the first inequality of the lemma is proved and the third inequality is proved since $C_{b2}' = \eta_\delta r_{22}$ from Lemma~\ref{lemma:cfirst:fuji}.
\end{proof}

Here, we focus on two terms related to $r_{21}{\upsilon_2^{cut}}$ in the upper bounds of $C_2$ and $C_{b2}$ in Lemma~\ref{lemma:Ccutupper2}. These terms have slower convergence rates than others, as seen later, and hence they are evaluated in a different way from others. Let
\begin{align}
\label{e:Cupper5}
\bar{C}_2^{>}=\bar{C}_{21}^{>}+\bar{C}_{22}^{>}, \qquad \bar{C}_{b2}^{>}=\bar{C}_{b21}^{>}+\bar{C}_{b22}^{>},
\end{align}
where
\begin{align*}
\bar{C}_{21}^{>} &= {\lambda_o} \left( c_g r_1 \sqrt{\upsilon^{cut}} + 4.8 \sqrt{e} \eta_4 r_{21} \frac{\upsilon_1^{cut}}{\sqrt{o}} \right), \\
\bar{C}_{22}^{>} &=4.8 \sqrt{e} \eta_4 {\lambda_o} r_{21} \frac{\upsilon_2^{cut}}{\sqrt{o}}, \\
\bar{C}_{b21}^{>} &=\sqrt{2\sigma^2} \eta_\delta r_{22} \left( c_g r_1 \sqrt{\frac{\upsilon^{cut}}{o}+1} + 4.8 \sqrt{e} \eta_4 r_{21} \left( \frac{\upsilon_1^{cut}}{o}+1 \right) \right), \\
\bar{C}_{b22}^{>} &=4.8 \sqrt{2e\sigma^2} \eta_\delta \eta_4 r_2\frac{\upsilon_2^{cut}}{o}.
\end{align*}

\begin{lemma} \label{l:slowterms} 
We have
\begin{align*}
\bar{C}_{22}^{>} &=4.8 \sqrt{e} \eta_4 {\lambda_o} r_{21} \frac{\upsilon_2^{cut}}{\sqrt{o}}  \le 9.6^2 e \eta_4 \frac{C_{\lambda_o}}{ C_{\lambda_o}^2-1 } \betamse, \\
\bar{C}_{b22}^{>} &=4.8 \sqrt{2e\sigma^2} \eta_\delta \eta_4 r_2 \frac{\upsilon_2^{cut}}{o} \le 9.6^2 e \eta_4 \frac{ 1 }{ C_{\lambda_o}^2-1 } \betamse.
\end{align*}
\end{lemma}
\begin{proof} 
We see 
\begin{align*}
\lambda_o \geq C_{\lambda_o}\sqrt{\frac{2\sigma^2\log (n/\delta)}{n}}\
= C_{\lambda_o} \sqrt{2\sigma^2} \eta_\delta \sqrt{\frac{\log n}{n}} 
=C_{\lambda_o} \sqrt{2\sigma^2} \eta_\delta \frac{1}{\sqrt{o}} r_{22}.
\end{align*}
We know $\upsilon_2^{cut}=C_{v2} \betamse/\lambda_o^2$ with $C_{v2}= 19.2\sqrt{2e\sigma^2}C_{r}$ from Lemma~\ref{l:Ccut_upper:fuji}. 
We also know $r_{21} \le r_{22}$ and $\eta_\delta \ge 1$ from the definition, and $C_r \le C_{\lambda_o}^2/(C_{\lambda_o}^2-1)$ from Lemma~\ref{lemma:Cr_const}. We see
\begin{align*}
\lambda_o r_{21}\frac{\upsilon_2^{cut}}{\sqrt{o}} &=  \lambda_o r_{21} \frac{1}{\sqrt{o}} \frac{C_{v2} \betamse}{\lambda_o^2} =  r_{21} C_{v2} \frac{1}{\sqrt{o}\lambda_o} \betamse 
\le \frac{ C_{v2} r_{21} }{ C_{\lambda_o} \sqrt{2\sigma^2} \eta_\delta r_{22} }  \betamse \\
&\le \frac{ 19.2\sqrt{e}C_{r} }{ C_{\lambda_o} \eta_\delta } \betamse 
\le \frac{ 19.2\sqrt{e} C_{\lambda_o} }{ C_{\lambda_o}^2-1 } \betamse 
\end{align*}
and
\begin{align*}
\sqrt{2\sigma^2}  \eta_\delta r_2\frac{\upsilon_2^{cut}}{o}
 &= \sqrt{2\sigma^2}\eta_\delta r_{22} r_{21} \frac{1}{o} \frac{C_{v2} \betamse}{\lambda_o^2} 
\le \frac{ \sqrt{2\sigma^2} \eta_\delta r_{22} r_{21} C_{v2} }{ C_{\lambda_o}^2 (2\sigma^2) \eta_\delta^2 r_{22}^2 }  \betamse\\
 &\le \frac{19.2\sqrt{e}C_{r}}{ C_{\lambda_o}^2 \eta_\delta } \betamse
\le \frac{19.2\sqrt{e}}{ C_{\lambda_o}^2-1 } \betamse.
\end{align*}
\end{proof}

\begin{proposition} \label{p:main:>o} 
We have
\begin{align*}
C_> \betamse \le ({\bar{C}_{01}} + \bar{C}_{02}^>) \nu_E  +  {\bar{C}_{b1}} + \bar{C}_{b21}^{>} + \bar{C}_{21}^{>} +{\bar{C}_3} + \lambda_s c_\kappa \sqrt{s}.
\end{align*}
\end{proposition}

\begin{proof}
We extract two terms $\bar{C}_{b22}^{>}$ and $\bar{C}_{22}^{>}$, which have slower convergence rates, from 
the basic inequality \eqref{e:CentralInequality:fuji}. From (c4) and $C_{\lambda_o}>1$, the corresponding L.H.S. of \eqref{e:CentralInequality:fuji} is expressed as
\begin{align*}
{\rm LHS}^- &=\frac{a_1^2}{2} \betamse - \bar{C}_{22}^{>} - \bar{C}_{b22}^{>}  
 \ge \left( \frac{9}{32} - 9.6^2 e \eta_4 \frac{ C_{\lambda_o} }{ C_{\lambda_o}^2-1 } - 9.6^2 e \eta_4 \frac{ 1 }{ C_{\lambda_o}^2-1 } \right) \betamse \\
& \ge \left( \frac{9}{32} - 2\times 9.6^2 e \bar{\eta}_4 \frac{ C_{\lambda_o} }{ C_{\lambda_o}^2-1 }  \right) \betamse = C_> \betamse.
\end{align*}
From the assumption of Theorem~\ref{t:main}, the coefficient of $\betamse$ is positive. 
From Proposition~\ref{p:DT19-1}, the corresponding R.H.S. of \eqref{e:CentralInequality:fuji} is given by
\begin{align*}
{\rm RHS}^- 
&=2 (C_{01}+C_{02}) +  C_{b}+C_2+C_3 +\lambda_s c_\kappa \sqrt{s} - \bar{C}_{22}^{>} - \bar{C}_{b22}^{>} \\
& \le 2(\bar{C}_{01}+\bar{C}_{01}^>) +  {\bar{C}_{b1}} + \bar{C}_{b21}^{>} + \bar{C}_{21}^{>} +{\bar{C}_3} + \lambda_s c_\kappa \sqrt{s}.
\end{align*}
From ${\rm RHS}^- \ge {\rm LHS}^-$, the proof is complete. 
\end{proof}

\subsection{Proof of Theorem~\ref{t:main}}

The case $C_{cut} \le o$ is proved by Proposition~\ref{p:main:<=o}.
The case $C_{cut} > o$ is proved by Proposition~\ref{p:main:>o}.


\subsection{Proof of Theorem~\ref{t:main2}}
First, we rewrite the upper bounds obtained above in the special case where
	\begin{align*}
	\lambda_o &=  C_{\lambda_o}\sqrt{\frac{2\sigma^2\log (n/\delta)}{n}},\quad	\lambda_s = \frac{4\sqrt{2}}{\sqrt{3}} C_{\lambda_s} \lambda_o.
\	\end{align*}

\begin{lemma} \label{l:O_lambda_o} 
\begin{align*}
\lambda_o &=  C_{\lambda_o}\sqrt{\frac{2\sigma^2\log (n/\delta)}{n}}=C_{\lambda_o}\sqrt{2\sigma^2} \eta_\delta  \frac{1}{\sqrt{o}} r_{22}. 
\end{align*}
\end{lemma}
\begin{proof} 
We have
\begin{align*}
\lambda_o &=  C_{\lambda_o}\sqrt{\frac{2\sigma^2\log (n/\delta)}{n}}=C_{\lambda_o}\sqrt{2\sigma^2} \eta_\delta \sqrt{\frac{\log n}{n}}.
\end{align*}
The second equality of the lemma holds from $r_{22}=\sqrt{o \log n/n}$. 
\end{proof}

\begin{lemma} \label{l:O_lambda_s}
We have
\begin{align*}
\lambda_s =\frac{1}{\sqrt{s}} O(r_{n,d,s,o}). 
\end{align*}
\end{lemma}

\begin{proof} 
We see
\begin{align*}
\lambda_o C_z = \lambda_o \sqrt{\frac{3 \rho^2 \sigma^2\log ({d}/{\delta}) }{\lambda_o^2 n} } 
 \le \sqrt{\frac{3 \rho^2 \sigma^2\log d }{n}} \sqrt{1+\log(1/\delta)} = \frac{1}{\sqrt{s}}O( r_1 )
\end{align*}
We have $g(o)=O(r_1+r_{21})$ from Lemmas~\ref{l:eta_const}~\ref{l:g1r}~and~\ref{l:lsog2:fuji}. 
We know $\lambda_o \sqrt{o}=O(r_{22})$ from Lemma~\ref{l:O_lambda_o} and $r_{22}=O(1)$ from Lemma~\ref{l:r_const}. Then we have
\begin{align*}
\lambda_o \sqrt{o} g(o) &= O(r_{22}) O(r_1+r_{21}) = O(r_{22} r_1+r_2)= O(r_1+r_2)=O(r_{n,d,s,o}).
\end{align*}
Hence, 
\begin{align*}
\lambda_s &= \frac{4\sqrt{2}}{\sqrt{3}} C_{\lambda_s} \lambda_o = \frac{4\sqrt{2}}{\sqrt{3}} \lambda_o \left(C_z+\sqrt{2 \frac{o}{s}} g(o) \right) = \frac{1}{\sqrt{s}} O(r_1) + \frac{1}{\sqrt{s}} O(r_{n,d,s,o}) = \frac{1}{\sqrt{s}} O(r_{n,d,s,o}).
\end{align*}
\end{proof}

\begin{lemma} \label{l:O_C} 
We have
\begin{align*}
\bar{C}_{01}=O(r_1), \quad \bar{C}_{b1}=O(r_1), \quad
\bar{C}_3 =O(r_{n,d,s,o}) 
\end{align*}
\end{lemma}

\begin{proof} 
We know $r_1=O(1)$ from Lemma~\ref{l:r_const}. Then we have
$\bar{C}_{01}=(1.2c_\kappa)^2 {2\rho^2} r_1^2 =O(r_1)$.
We see $\bar{C}_{b1}=2c_\kappa \sqrt{2\sigma^2\rho^2} \sqrt{1+\log(1/\delta)} r_1 = O(r_1)$. 
From Lemmas~\ref{l:O_lambda_o},~\ref{l:eta_const}~and~\ref{l:r_const}, we see
\begin{align*}
\bar{C}_3 &= (c_g r_1  + 4.8\sqrt{e}\eta_4 r_{21} )  \sqrt{o} \lambda_o = O(r_1+r_{21})O(r_{22})=O(r_1r_{22}+r_2)
=O(r_{n,d,s,o}).
\end{align*}
\end{proof}

\begin{lemma} \label{l:O_nuE} 
We have
\begin{align*}
\nu_E &=\frac{6}{C_{n,\delta}^2}\sqrt{\frac{\lambda_s^2 s}{\kappa^2} + 6.25 \lambda^2_o o}=O(r_1+r_{22}).
\end{align*}
\end{lemma}

\begin{proof} 
We know $1/C_{n,\delta}=O(1)$ from Lemma~\ref{l:Cndelta}, 
$\sqrt{o} \lambda_o =O(r_{22})$ from Lemma~\ref{l:O_lambda_o}, 
$\sqrt{s} \lambda_s =O(r_1+r_{21}r_{22})$ from Lemma~\ref{l:O_lambda_s}, and $r_{21}=O(1)$ from Lemma~\ref{l:r_const}. 
Hence, by $\sqrt{A+B} \le \sqrt{A}+\sqrt{B}$ for $A,B>0$, 
\begin{align*}
\nu_E \le \frac{6}{C_{n,\delta}^2} \left( \frac{\lambda_s \sqrt{s}}{\kappa} + \sqrt{6.25} \lambda_o \sqrt{o} \right)=O(r_1+r_{22}).
\end{align*}
\end{proof}

\begin{lemma} \label{l:O_C<} 
We have
\begin{align*}
\bar{C}_{02}^{<} =O(r_{n,d,s,o}), \quad 
\bar{C}_{b2}^{<} =O(r_{n,d,s,o}), \quad 
\bar{C}_{2}^{<} =O(r_{n,d,s,o}).
\end{align*}
\end{lemma}

\begin{proof} 
We know $\eta_4=O(1)$ and $\eta_\delta=O(1)$ from Lemma~\ref{l:eta_const} and $r_1=O(1)$ and $r_{22}=O(1)$ from Lemma~\ref{l:r_const}. We also know $\lambda_o \sqrt{o}=O(r_{22})$ from Lemma~\ref{l:O_lambda_o}. Hence,
\begin{align*}
\bar{C}_{02}^{<} &=2 (c_g)^2 r_1^2 + 2\times 4.8^2 e\eta_4^2r_2 =O(r_1+r_2)=O(r_{n,d,s,o}), \\
\bar{C}_{b2}^{<} &=2 c_g \sqrt{ \sigma^2 } \eta_\delta r_1 r_{22} + 9.6 \sqrt{2e\sigma^2} \eta_4 \eta_\delta r_{2} =O(r_1+r_2)=O(r_{n,d,s,o}), \\
\bar{C}_{2}^{<} &=  \lambda_o \sqrt{o} (c_gr_1+4.8\sqrt{e}\eta_4 r_{21}) = O(r_1r_{22}+r_{21}r_{22}) = O(r_1+r_2)=O(r_{n,d,s,o}).
\end{align*}
\end{proof}

\begin{proposition} 
In the case $C_{cut} \le o$, we have $\betamse=O(r_{n,d,s,o})$. 
\end{proposition}
\begin{proof}
Each term of the upper bound of $\betamse$ in \eqref{p:main:<=o} is shown to be $O(r_{n,d,s,o})$ from Lemmas~\ref{l:O_C}, \ref{l:O_C<}, \ref{l:O_nuE}, \ref{l:r_const} and \ref{l:O_lambda_s}. The proof is complete. 
\end{proof}

\begin{lemma} \label{l:O_C>_2}
We have
\begin{align*}
\bar{C}_{02}^> \nu_E =O(r_{n,d,s,o}), \qquad \bar{C}_{b21}^> =O(r_{n,d,s,o}), \qquad \bar{C}_{21}^> =O(r_{n,d,s,o}).
\end{align*}
\end{lemma}

\begin{proof}
From Lemma~\ref{l:slowterms}, we know
\begin{align*}
\lambda_o =C_{\lambda_o} \sqrt{2\sigma^2} \eta_\delta \frac{1}{\sqrt{o}} r_{22}.
\end{align*}
and then $\lambda_o=O(r_{22}/\sqrt{o})$ and $1/\lambda_o=O(1)\sqrt{o}/r_{22}$. Hence, from Lemmas~\ref{l:O_lambda_s},~\ref{l:O_nuE}~and~\ref{l:r_const},
\begin{align*}
\upsilon^{cut}_1 &= 2C_r \left( \frac{\sqrt{2\sigma^2}}{\lambda_o^2} c_g r_1 + \frac{\sqrt{o}}{\lambda_o} c_g r_1 + \frac{\sqrt{o}}{\lambda_o} 4.8\sqrt{e} \eta_4 r_{21} +\sqrt{s}c_\kappa\frac{\lambda_s}{\lambda_o^2} \right)\betamse \\
&= \left( \frac{o}{r_{22}^2} O(r_1) + \frac{o}{r_{22}} O(r_1+r_{21}) + O(r_{n,d,s,o}) \frac{o}{r_{22}^2} \right) O(r_1+r_{22}) \\
&= \frac{o}{r_{22}^2} O(r_{n,d,s,o}), \\
\upsilon_2^{cut} &=  C_{v2} \frac{1}{\lambda_o^2} \betamse = \frac{o}{r_{22}^2} O(r_1+r_{22}).
\end{align*}
Then, from $\sqrt{A+B} \le \sqrt{A}+\sqrt{B}$ for $A,B>0$, 
\begin{align*}
\sqrt{\frac{\upsilon^{cut}}{o}+1 } \le \sqrt{\frac{\upsilon_1^{cut}}{o}}+\sqrt{\frac{\upsilon_2^{cut}}{o}}+1 = 1+\frac{1}{r_{22}} O(\sqrt{r_1}+\sqrt{r_{22}}).
\end{align*}
Hence, from Lemmas~\ref{l:eta_const}~and~\ref{l:r_const}, and $r_{21}\le r_{22}$, we see
\begin{align*}
\bar{C}_{02}^{>} &=(c_g)^2 r_1^2 + 9.6 c_g \sqrt{e} \eta_4 r_1 r_{21}\sqrt{\frac{\upsilon^{cut}}{o}+1 }  
+ 4.8^2 e \eta_4^2 r_{21}^2 \left( \frac{\upsilon^{cut}}{o}+1 \right) \\
 &= O(r_1^2) + O(r_1r_{21}) + \frac{r_1r_{21}}{r_{22}} O(\sqrt{r_1}+\sqrt{r_{22}}) + \frac{r_{21}^2}{r_{22}^2} O(r_1+r_{22}) \\
 &=O(r_1) + \frac{r_{21}^2}{r_{22}} O(1), \\
\bar{C}_{02}^{>} \nu_E &=\left( O(r_1) + \frac{r_{21}^2}{r_{22}} O(1) \right)O(r_1+r_{22}) \\
 &=O(r_1)+O(r_{21}^2)=O(r_1)+O(r_{21}r_{22})=O(r_{n,d,s,o}), \\
\bar{C}_{b21}^{>} &=\sqrt{2\sigma^2} \eta_\delta r_{22} \left( c_g r_1 \sqrt{\frac{\upsilon^{cut}}{o}+1} + 4.8 \sqrt{e} \eta_4 r_{21} \left( \frac{\upsilon_1^{cut}}{o}+1 \right) \right) \\
 &= O( r_{22} ) \left\{ r_1 \left( 1+\frac{1}{r_{22}} O(\sqrt{r_1}+\sqrt{r_{22}}) \right)  + O(r_{21}) \left( 1 + \frac{1}{r_{22}^2}O(r_{n,d,s,o}) \right) \right\} \\
 &= O( r_1(r_{22}+\sqrt{r_1}+\sqrt{r_{22}}))  + O(r_{22}r_{21})+O(r_{n,d,s,o}) = O(r_{n,d,s,o}), \\
\bar{C}_{21}^{>} &= {\lambda_o} \left( c_g r_1 \sqrt{\upsilon^{cut}} + 4.8 \sqrt{e} \eta_4 r_{21} \frac{\upsilon_1^{cut}}{\sqrt{o}} \right), \\
 & = O\left( \frac{r_{22}}{\sqrt{o}} \right) \left\{ r_1 \frac{\sqrt{o}}{r_{22}} O\left(  \sqrt{r_1}+\sqrt{r_{22}} \right) + r_{21} \frac{\sqrt{o}}{r_{22}^2} O(r_{n,d,s,o})\right\} \\
 & = O( r_1 )  + O(r_{n,d,s,o}) =O(r_{n,d,s,o}). 
\end{align*}
\end{proof}

\begin{proposition} 
In the case $C_{cut} > o$, we have $\betamse=O(r_{n,d,s,o})$. 
\end{proposition}
\begin{proof}
Each term of the upper bound of $\betamse$ in \eqref{eq2:t:main} is shown to be $O(r_{n,d,s,o})$ from Lemmas~\ref{l:O_C}, \ref{l:O_nuE}, \ref{l:r_const}, \ref{l:O_lambda_s} and \ref{l:O_C>_2}. The proof is complete. 
\end{proof}


\appendix
\section{Proof of Proposition~\ref{p:XxiBernstein}}

Here we give the Bernstein concentration inequality.
\begin{theorem}[Bernstein concentration inequality]
	\label{th:Bernstein}
	Let $\{Z_i\}_{i=1}^n$ be a sequence with i.i.d random variables. 
	We assume that
	\begin{align*}
	\sum_{i=1}^n\mr{E}[X_i^2] \leq v,\ \sum_{i=1}^n\mr{E}[(X_i)_+^k] \leq \frac{k!}{2}vc^{k-2}
	\end{align*}
	for $i=1,\cdots n$ and for $k \in \mbb{N}$ such that $k\geq 3$. Then, we have
	\begin{align*}
	\mr{P}\left[\left|\sum_{i=1}^n (X_i - \mr{E}[X_i]) \right| \leq\sqrt{2vt}+ct \right] \geq 1-e^{-t}
	\end{align*}
	for any $t>0$.
\end{theorem}
Using Theorem~\ref{th:Bernstein}, we can prove Proposition~\ref{p:XxiBernstein}, which is given in the following.
\begin{proposition}
	\label{p:fBernstein}
	Let $\{\xi_i\}_{i=1}^n$ be a sequence with i.i.d random variables drawn from $\mc{N}(0,\sigma^2)$ and 
	$\{X_i\}_{i=1}^n$  drawn from $\mc{N}(0,\Sigma)$. 
	Let $z_{ij}= X_{ij} \psi\left(\frac{\xi_i}{\lambda_o\sqrt{n}}\right)$ and $z = (\sum_{i=1}^nz_{i1}, \cdots ,\sum_{i=1}^nz_{d,1})$.
	For any $\delta \in (0,1)$ and $n$ such that $\sqrt{\frac{\log (d/ \delta)}{n}} \leq \sqrt{3} -\sqrt{2}$, with probability at least $1-\delta$, we have
		\begin{align*}
		\left \|\frac{z}{\sqrt{n}} \right\|_\infty \leq \sqrt{3\frac{\rho^2 \sigma^2}{n\lambda_o^2}\log \frac{d}{\delta}}=:C_z.
		\end{align*}
\end{proposition}
\begin{proof}
	We have $\mr{E}[z_{ij}] = 0$. Since $|\psi(t)| \le |t|$, we see
	\begin{align*}
	\sum_{i=1}^n \mr{E}[z_{ij}^2]
	= \sum_{i=1}^n \mr{E}[X_{ij}^2] \mr{E}\left[\psi\left(\frac{\xi_j}{\lambda_o \sqrt{n}}\right)^2 \right]
	\leq \sum_{i=1}^n \mr{E}[X_{ij}^2] \mr{E}\left[\frac{\xi_j^2}{\lambda_o^2 n} \right]
	\le \frac{\rho^2 \sigma^2}{\lambda_o^2}.
	\end{align*}
	From Proposition 3.2. of \cite{Riv2012Subgaussian}, we can show that the absolute $k (\ge 3)$th moment of $z_{ij}$ is bounded above, as follows: 
	\begin{align*}
	\sum_{i=1}^n \mr{E}[|z_{ij}|^k] \leq \sum_{i=1}^n \mr{E}[|X_{ij}|^k] \mr{E}\left[\left| \frac{\xi_j}{\lambda_o^2 n} \right|^k\right] \leq \frac{k!}{2} \frac{\rho^2 \sigma^2}{\lambda_o^2} \left(\frac{\rho^2 \sigma^2}{\lambda_o^2n }\right)^{\frac{k-2}{2}}.
	\end{align*}
	From Theorem \ref{th:Bernstein} with $t = \log(d/\delta)$, $v=\frac{\rho^2 \sigma^2}{\lambda_o^2}$ and $c=\sqrt{\frac{\rho^2 \sigma^2}{\lambda_o^2n }}$, we have
	\begin{align*}
	\mr{P}\left[\left|\sum_{i=1}^nz_{ij} \right| \leq\sqrt{2\frac{\rho^2 \sigma^2}{\lambda_o^2} \log(d/\delta)}+\sqrt{\frac{\rho^2 \sigma^2}{\lambda_o^2 n}} \log(d/\delta) \right] \geq 1-\delta/d.
	\end{align*}
	By the condition $\sqrt{\frac{\log d/\delta}{n}} \leq \sqrt{3} -\sqrt{2}$, the above inequality is
	\begin{align*}
	\mr{P}\left[\left|\sum_{i=1}^nz_{ij} \right| \leq \sqrt{n} C_z \right] \geq 1-\delta/d.
	\end{align*}
	Hence, 
	\begin{align*}
	\mr{P}\left[\left\|z \right\|_\infty \leq  \sqrt{n} C_z\right] 
	&= \mr{P}\left[\sup_j \left|\sum_{i=1}^n z_{ij} \right|  \leq  \sqrt{n} C_z\right] = 1- \mr{P}\left[\sup_j \left|\sum_{i=1}^n z_{ij} \right| > \sqrt{n} C_z\right] \nonumber \\
	&=1- \mr{P}\left[\bigcup_j \left\{ \left|\sum_{i=1}^n z_{ij} \right|  >  \sqrt{n} C_z\right\} \right] 
	\ge 1 - \sum_{j=1}^d \mr{P}\left[ \left|\sum_{i=1}^n z_{ij} \right|  >  \sqrt{n} C_z \right] \nonumber \\
	&\ge 1-(\delta/d)d = 1-\delta. 
	\end{align*}
\end{proof}

\section{Proof of Lemma~\ref{l:DT19-7-2}}
\label{appendix:Lemma7DT}

We give Lemma~\ref{l:DT19-7-2} again in the following.

\begin{lemma}[Modification of Lemma 7 of \cite{DalTho2019Outlier}]
	\label{l:DT19-7}
	Let $Z \in \mbb{R}^{n\times d}$ be a random matrix satisfying 
	\begin{align}
	\left\|\frac{Z}{\sqrt{n}}v \right\|_2 \geq a_1 \|\Sigma^{\frac{1}{2}} v \|_2 -a_2\|v\|_1 \label{e:asd1}
	\end{align}
	and
	\begin{align}
	\left|u^\top\frac{Z}{\sqrt{n}}v\right| \leq b_1 \|\Sigma^{\frac{1}{2}} v\|_2\|u\|_2 + b_2\|c\|_1\|u\|_2 + b_3\|\Sigma^{\frac{1}{2}}v\|_2\|u\|_1 \label{e:asd2}
	\end{align}
	for some positive constants $a_1 \in (0,1),\ a_2,\ b_1,\ b_2,\ b_3$. Then, for any $\alpha>0$, $Z$ satisfies
	\begin{align*}
	\left\|\frac{Z}{\sqrt{n}}v+u\right\|_2 \geq c_1 \left(\|\Sigma^{\frac{1}{2}}v\|_2 +\|u\|_2\right)-c_2\|v\|_1-c_3\|u\|_1
	\end{align*}
	with the constants $c_1=\sqrt{a_1^2+b_1 + \alpha^2}-\sqrt{2(b_1+\alpha^2)},\ c_2 = a_2+b_2/\alpha,\ c_3 = b_3/\alpha$.
	If $a_1^2 > b_1+ \alpha^2$, then we have $c_1>0$.
\end{lemma}

\begin{proof}
	From (\ref{e:asd1}) and simple calculation, 
	\begin{align*}
	&\sqrt{a_1^2+b_1+\alpha^2}\left\{\|\Sigma^{1/2}v\|_2^2 + \|u\|_2^2\right\}^{1/2}\\ &= \left\{ a_1^2\|\Sigma^{1/2}v\|_2^2 + a_1^2\|u\|_2^2+(b_1+\alpha^2)(\|\Sigma^{1/2}v\|_2^2+\|u\|_2^2)\right\}^{1/2} \\
	&\leq \left\{\left( \left\|\frac{Z}{\sqrt{n}}v\right\|_2+a_2\|v\|_1\right)^2+a_1^2\|u\|_2^2+(b_1+\alpha^2)(\|\Sigma^{1/2}v\|_2^2+\|u\|_2^2) \right\}^{1/2}\\
	&\leq \left\{ \left\|\frac{Z}{\sqrt{n}}v\right\|_2^2+\|u\|_2^2+(b_1+\alpha^2)(\|\Sigma^{1/2}v\|_2^2+\|u\|_2^2)\right\}^{1/2}+a_2\|v\|_1.
	\end{align*}
	From Young's inequality, we know $uv \le (\gamma/2)u^2+(1/2\gamma)v^2$ for $\gamma>0$. 
	Using this inequality and \ref{e:asd2}, we see
	\begin{align*}
	\left\|\frac{Z}{\sqrt{n}}v\right\|_2^2+\|u\|_2^2 &=\left\|\frac{Z}{\sqrt{n}}v+u\right\|_2^2-2u^\top \frac{Z}{\sqrt{n}}v\\
	&\leq\left\|\frac{Z}{\sqrt{n}}v+u\right\|_2^2+2b_1\|\Sigma^{1/2}v\|_2\|u\|_2+2b_2\|v\|_1\|u\|_2+2b_3\|\Sigma^{1/2}v\|_2\|u\|_1 \\
	&\leq \left\|\frac{Z}{\sqrt{n}}v+u\right\|_2^2+(b_1+\alpha^2)\left(\|\Sigma^{1/2}v\|_2^2+\|u\|_2^2\right)+\frac{b_2^2}{\alpha^2}\|v\|_1^2+\frac{b_3^2}{\alpha^2}\|u\|_1^2
	\end{align*}
	Combining the above two properties,
	\begin{align*}
	&\sqrt{a_1^2+b_1+\alpha^2}\left\{\|\Sigma^{1/2}v\|_2^2 + \|u\|_2^2\right\}^{1/2}\\
	&\leq \left\{\left\|\frac{Z}{\sqrt{n}}v+u\right\|_2^2+2(b_1+\alpha^2)(\|\Sigma^{1/2}v\|_2^2+\|u\|_2^2) +\frac{b_2^2}{\alpha^2}\|v\|_1^2+\frac{b_3^2}{\alpha^2}\|u\|_1^2\right\}^{1/2} +a_2\|v\|_1 \\
	& \leq \left\|\frac{Z}{\sqrt{n}}v+u\right\|_2+ \sqrt{2(b_1+\alpha^2)}
	\left\{\|\Sigma^{1/2}v\|_2^2 + \|u\|_2^2\right\}^{1/2}+\frac{b_2}{\alpha}\|v\|_1 + \frac{b_3}{\alpha}\|u\|_1+a_2\|v\|_1 
	\end{align*}
	Rearranging the terms, 
	\begin{align*}
	&\left(\sqrt{a_1^2+b_1 + \alpha^2}-\sqrt{2(b_1+\alpha^2)}\right)
	\left\{\|\Sigma^{1/2}v\|_2^2 + \|u\|_2^2\right\}^{1/2} \\
	&\leq\left\|\frac{Z}{\sqrt{n}}v+u\right\|_2+\left(\frac{b_2}{\alpha}+a_2\right)\|v\|_1 + \frac{b_3}{\alpha}\|u\|_1.
	\end{align*}
	The condition $a_1^2>b_1 + \alpha^2$ implies $c_1>0$.
\end{proof}

\section{Condition (\ref{cond0})}
\label{sec:cond0}
We investigate the condition (\ref{cond0}) in detail. We assume the conditions used in Theorem~\ref{t:main2}.
As seen in Lemma~\ref{l:Cndelta}, the R.H.S. of \eqref{cond0}, $C_{n,\delta}$, is bounded above 0. We will show that the L.H.S. of \eqref{cond0} can be sufficiently small under some conditions, so that the condition \eqref{cond0} is satisfied. 

Let
\begin{align*}
A_1=\frac{\log d}{n}, \quad A_2=\frac{\lambda_s^2}{\lambda_o^2} {\frac{\log n}{n}}, \quad 
B_1=s, \quad B_2=o \frac{\lambda_o^2}{\lambda_s^2}.
\end{align*}
The L.H.S. of \eqref{cond0} is bounded up to constant by the square root of $A_1B_1+A_1B_2+A_2B_1+A_2B_2$. Hereafter, we evaluate each term. We see
\begin{align*}
A_1B_1 = s^2{\frac{\log d}{n}}=r_1^2, \qquad
A_2B_2 =o {\frac{\log n}{n}}=r_{22}^2.
\end{align*}
From Lemmas~\ref{l:O_lambda_o}~and~\ref{l:O_lambda_s} and  $\eta_\delta \ge 1$, 
\begin{align*}
A_2B_1 =s \frac{\lambda_s^2}{\lambda_o^2} \frac{\log n}{n} = 
s \left\{ \frac{1}{\eta_\delta r_{22}/\sqrt{o}} \frac{O(r_{n,d,s,o})}{\sqrt{s}} \right\}^2 \frac{\log n}{n} = O(r_{n,d,s,o}^2).
\end{align*}
Here, we see
\begin{align*}
\frac{\lambda_s}{\lambda_o} 
&= \frac{4\sqrt{2}}{\sqrt{3}} C_{\lambda_s} \ge \frac{4\sqrt{2}}{\sqrt{3}} C_z
 = \frac{4\sqrt{2}}{\sqrt{3}} \frac{3\rho^2\sigma^2 \log(d/\delta)}{\lambda_o^2 n} \\
& = \frac{4\sqrt{2}}{\sqrt{3}} \frac{3\rho^2\sigma^2 \log(d/\delta)}{C_{\lambda_o}^2 2\sigma^2 \log(n/\delta)} = \frac{2\sqrt{6}\rho^2}{C_{\lambda_o}^2} \frac{\log(d/\delta)}{\log(n/\delta)},
\end{align*}
Hence, 
\begin{align*}
A_1B_2 &= \frac{\log d}{n} o \frac{\lambda_o^2}{\lambda_s^2} \le \frac{C_{\lambda_o}^2}{2\sqrt{6}\rho^2} \frac{\log 3d}{n} o \frac{\log(n/\delta)}{\log(d/\delta)} \\
&= \frac{C_{\lambda_o}^2}{2\sqrt{6}\rho^2}\eta_\delta^2 r_{22}^2 \frac{\log d}{\log(d/\delta)} =O(r_{22}^2).
\end{align*}
because $\eta_\delta$ is bounded above from Lemma~\ref{l:eta_const} and ${\log d}/{\log(d/\delta)} \le 1$ from $\delta \in (0,1/7)$. 
Therefore, if $r_1$ and $r_{22}$ are sufficiently small, then the L.H.S. of \eqref{cond0} is sufficiently small, so that \eqref{cond0} is satisfied. 


\bibliographystyle{imsart-nameyear}
\bibliography{RELOAC}

\begin{thebibliography}{24}

\bibitem[\protect\citeauthoryear{Bellec, Lecu{\'e} and
  Tsybakov}{2018}]{BelLecTsy2018Slope}
\begin{barticle}[author]
\bauthor{\bsnm{Bellec},~\bfnm{Pierre~C}\binits{P.~C.}},
  \bauthor{\bsnm{Lecu{\'e}},~\bfnm{Guillaume}\binits{G.}} \AND
  \bauthor{\bsnm{Tsybakov},~\bfnm{Alexandre~B}\binits{A.~B.}}
(\byear{2018}).
\btitle{Slope meets lasso: improved oracle bounds and optimality}.
\bjournal{The Annals of Statistics}
\bvolume{46}
\bpages{3603--3642}.
\end{barticle}
\endbibitem

\bibitem[\protect\citeauthoryear{Boucheron, Lugosi and
  Massart}{2013}]{BouLugMas2013concentration}
\begin{bbook}[author]
\bauthor{\bsnm{Boucheron},~\bfnm{St{\'e}phane}\binits{S.}},
  \bauthor{\bsnm{Lugosi},~\bfnm{G{\'a}bor}\binits{G.}} \AND
  \bauthor{\bsnm{Massart},~\bfnm{Pascal}\binits{P.}}
(\byear{2013}).
\btitle{Concentration inequalities: A nonasymptotic theory of independence}.
\bpublisher{Oxford university press}.
\end{bbook}
\endbibitem

\bibitem[\protect\citeauthoryear{Candes et~al.}{2007}]{CanTao2007Dantzig}
\begin{barticle}[author]
\bauthor{\bsnm{Candes},~\bfnm{Emmanuel}\binits{E.}},
  \bauthor{\bsnm{Tao},~\bfnm{Terence}\binits{T.}} \betal{et~al.}
(\byear{2007}).
\btitle{The Dantzig selector: Statistical estimation when p is much larger than
  n}.
\bjournal{The Annals of Statistics}
\bvolume{35}
\bpages{2313--2351}.
\end{barticle}
\endbibitem

\bibitem[\protect\citeauthoryear{Chen, Caramanis and
  Mannor}{2013}]{CheCarMan2013Robust}
\begin{binproceedings}[author]
\bauthor{\bsnm{Chen},~\bfnm{Yudong}\binits{Y.}},
  \bauthor{\bsnm{Caramanis},~\bfnm{Constantine}\binits{C.}} \AND
  \bauthor{\bsnm{Mannor},~\bfnm{Shie}\binits{S.}}
(\byear{2013}).
\btitle{Robust sparse regression under adversarial corruption}.
In \bbooktitle{International Conference on Machine Learning}
\bpages{774--782}.
\end{binproceedings}
\endbibitem

\bibitem[\protect\citeauthoryear{Chen, Gao and Ren}{2018}]{CheGaoRen2018robust}
\begin{barticle}[author]
\bauthor{\bsnm{Chen},~\bfnm{Mengjie}\binits{M.}},
  \bauthor{\bsnm{Gao},~\bfnm{Chao}\binits{C.}} \AND
  \bauthor{\bsnm{Ren},~\bfnm{Zhao}\binits{Z.}}
(\byear{2018}).
\btitle{Robust covariance and scatter matrix estimation under Huber's
  contamination model}.
\bjournal{The Annals of Statistics}
\bvolume{46}
\bpages{1932--1960}.
\end{barticle}
\endbibitem

\bibitem[\protect\citeauthoryear{Cheng, Diakonikolas and
  Ge}{2019}]{CheFiaGe2019High}
\begin{binproceedings}[author]
\bauthor{\bsnm{Cheng},~\bfnm{Yu}\binits{Y.}},
  \bauthor{\bsnm{Diakonikolas},~\bfnm{Ilias}\binits{I.}} \AND
  \bauthor{\bsnm{Ge},~\bfnm{Rong}\binits{R.}}
(\byear{2019}).
\btitle{High-dimensional robust mean estimation in nearly-linear time}.
In \bbooktitle{Proceedings of the Thirtieth Annual ACM-SIAM Symposium on
  Discrete Algorithms}
\bpages{2755--2771}.
\bpublisher{SIAM}.
\end{binproceedings}
\endbibitem

\bibitem[\protect\citeauthoryear{Dalalyan and
  Thompson}{2019}]{DalTho2019Outlier}
\begin{bincollection}[author]
\bauthor{\bsnm{Dalalyan},~\bfnm{Arnak}\binits{A.}} \AND
  \bauthor{\bsnm{Thompson},~\bfnm{Philip}\binits{P.}}
(\byear{2019}).
\btitle{Outlier-robust estimation of a sparse linear model using
  $\ell_1$-penalized Huber's M-estimator}.
In \bbooktitle{Advances in Neural Information Processing Systems 32}
(\beditor{\bfnm{H.}\binits{H.}~\bsnm{Wallach}},
  \beditor{\bfnm{H.}\binits{H.}~\bsnm{Larochelle}},
  \beditor{\bfnm{A.}\binits{A.}~\bsnm{Beygelzimer}},
  \beditor{\bfnm{F.}\binits{F.}~\bparticle{d'Alch\'{e}} \bsnm{Buc}},
  \beditor{\bfnm{E.}\binits{E.}~\bsnm{Fox}} \AND
  \beditor{\bfnm{R.}\binits{R.}~\bsnm{Garnett}}, eds.)
\bpages{13188--13198}.
\bpublisher{Curran Associates, Inc.}
\end{bincollection}
\endbibitem

\bibitem[\protect\citeauthoryear{Diakonikolas and
  Kane}{2019}]{DiaKan2019Recent}
\begin{barticle}[author]
\bauthor{\bsnm{Diakonikolas},~\bfnm{Ilias}\binits{I.}} \AND
  \bauthor{\bsnm{Kane},~\bfnm{Daniel~M}\binits{D.~M.}}
(\byear{2019}).
\btitle{Recent Advances in Algorithmic High-Dimensional Robust Statistics}.
\bjournal{arXiv preprint arXiv:1911.05911}.
\end{barticle}
\endbibitem

\bibitem[\protect\citeauthoryear{Diakonikolas, Kong and
  Stewart}{2019}]{DiaKonSte2019Efficient}
\begin{binproceedings}[author]
\bauthor{\bsnm{Diakonikolas},~\bfnm{Ilias}\binits{I.}},
  \bauthor{\bsnm{Kong},~\bfnm{Weihao}\binits{W.}} \AND
  \bauthor{\bsnm{Stewart},~\bfnm{Alistair}\binits{A.}}
(\byear{2019}).
\btitle{Efficient algorithms and lower bounds for robust linear regression}.
In \bbooktitle{Proceedings of the Thirtieth Annual ACM-SIAM Symposium on
  Discrete Algorithms}
\bpages{2745--2754}.
\bpublisher{SIAM}.
\end{binproceedings}
\endbibitem

\bibitem[\protect\citeauthoryear{Diakonikolas
  et~al.}{2017}]{DiaKamKanLiMotSte2017Being}
\begin{binproceedings}[author]
\bauthor{\bsnm{Diakonikolas},~\bfnm{Ilias}\binits{I.}},
  \bauthor{\bsnm{Kamath},~\bfnm{Gautam}\binits{G.}},
  \bauthor{\bsnm{Kane},~\bfnm{Daniel~M}\binits{D.~M.}},
  \bauthor{\bsnm{Li},~\bfnm{Jerry}\binits{J.}},
  \bauthor{\bsnm{Moitra},~\bfnm{Ankur}\binits{A.}} \AND
  \bauthor{\bsnm{Stewart},~\bfnm{Alistair}\binits{A.}}
(\byear{2017}).
\btitle{Being robust (in high dimensions) can be practical}.
In \bbooktitle{Proceedings of the 34th International Conference on Machine
  Learning-Volume 70}
\bpages{999--1008}.
\bpublisher{JMLR}.
\end{binproceedings}
\endbibitem

\bibitem[\protect\citeauthoryear{Diakonikolas
  et~al.}{2018}]{DiaKamKanLiMoiSte2018Robustly}
\begin{binproceedings}[author]
\bauthor{\bsnm{Diakonikolas},~\bfnm{Ilias}\binits{I.}},
  \bauthor{\bsnm{Kamath},~\bfnm{Gautam}\binits{G.}},
  \bauthor{\bsnm{Kane},~\bfnm{Daniel~M}\binits{D.~M.}},
  \bauthor{\bsnm{Li},~\bfnm{Jerry}\binits{J.}},
  \bauthor{\bsnm{Moitra},~\bfnm{Ankur}\binits{A.}} \AND
  \bauthor{\bsnm{Stewart},~\bfnm{Alistair}\binits{A.}}
(\byear{2018}).
\btitle{Robustly learning a gaussian: Getting optimal error, efficiently}.
In \bbooktitle{Proceedings of the Twenty-Ninth Annual ACM-SIAM Symposium on
  Discrete Algorithms}
\bpages{2683--2702}.
\bpublisher{SIAM}.
\end{binproceedings}
\endbibitem

\bibitem[\protect\citeauthoryear{Diakonikolas
  et~al.}{2019a}]{DiaKamKanLiMoiSte2019Robust}
\begin{barticle}[author]
\bauthor{\bsnm{Diakonikolas},~\bfnm{Ilias}\binits{I.}},
  \bauthor{\bsnm{Kamath},~\bfnm{Gautam}\binits{G.}},
  \bauthor{\bsnm{Kane},~\bfnm{Daniel}\binits{D.}},
  \bauthor{\bsnm{Li},~\bfnm{Jerry}\binits{J.}},
  \bauthor{\bsnm{Moitra},~\bfnm{Ankur}\binits{A.}} \AND
  \bauthor{\bsnm{Stewart},~\bfnm{Alistair}\binits{A.}}
(\byear{2019}a).
\btitle{Robust estimators in high-dimensions without the computational
  intractability}.
\bjournal{SIAM Journal on Computing}
\bvolume{48}
\bpages{742--864}.
\end{barticle}
\endbibitem

\bibitem[\protect\citeauthoryear{Diakonikolas
  et~al.}{2019b}]{DiaKanKarPriSte2019Outlier}
\begin{binproceedings}[author]
\bauthor{\bsnm{Diakonikolas},~\bfnm{Ilias}\binits{I.}},
  \bauthor{\bsnm{Kane},~\bfnm{Daniel}\binits{D.}},
  \bauthor{\bsnm{Karmalkar},~\bfnm{Sushrut}\binits{S.}},
  \bauthor{\bsnm{Price},~\bfnm{Eric}\binits{E.}} \AND
  \bauthor{\bsnm{Stewart},~\bfnm{Alistair}\binits{A.}}
(\byear{2019}b).
\btitle{Outlier-Robust High-Dimensional Sparse Estimation via Iterative
  Filtering}.
In \bbooktitle{Advances in Neural Information Processing Systems}
\bpages{10688--10699}.
\end{binproceedings}
\endbibitem

\bibitem[\protect\citeauthoryear{Fan and Li}{2001}]{Fan2001Variable}
\begin{barticle}[author]
\bauthor{\bsnm{Fan},~\bfnm{Jianqing}\binits{J.}} \AND
  \bauthor{\bsnm{Li},~\bfnm{Runze}\binits{R.}}
(\byear{2001}).
\btitle{Variable selection via nonconcave penalized likelihood and its oracle
  properties}.
\bjournal{Journal of the American statistical Association}
\bvolume{96}
\bpages{1348--1360}.
\end{barticle}
\endbibitem

\bibitem[\protect\citeauthoryear{Gao}{2020}]{Gao2020Robust}
\begin{barticle}[author]
\bauthor{\bsnm{Gao},~\bfnm{Chao}\binits{C.}}
(\byear{2020}).
\btitle{Robust regression via mutivariate regression depth}.
\bjournal{Bernoulli}
\bvolume{26}
\bpages{1139--1170}.
\end{barticle}
\endbibitem

\bibitem[\protect\citeauthoryear{Lai, Rao and
  Vempala}{2016}]{LaiRaoVem2016Agnostic}
\begin{binproceedings}[author]
\bauthor{\bsnm{Lai},~\bfnm{Kevin~A}\binits{K.~A.}},
  \bauthor{\bsnm{Rao},~\bfnm{Anup~B}\binits{A.~B.}} \AND
  \bauthor{\bsnm{Vempala},~\bfnm{Santosh}\binits{S.}}
(\byear{2016}).
\btitle{Agnostic estimation of mean and covariance}.
In \bbooktitle{2016 IEEE 57th Annual Symposium on Foundations of Computer
  Science (FOCS)}
\bpages{665--674}.
\bpublisher{IEEE}.
\end{binproceedings}
\endbibitem

\bibitem[\protect\citeauthoryear{Laurent and
  Massart}{2000}]{LauMas2000Adaptive}
\begin{barticle}[author]
\bauthor{\bsnm{Laurent},~\bfnm{Beatrice}\binits{B.}} \AND
  \bauthor{\bsnm{Massart},~\bfnm{Pascal}\binits{P.}}
(\byear{2000}).
\btitle{Adaptive estimation of a quadratic functional by model selection}.
\bjournal{The Annals of Statistics}
\bpages{1302--1338}.
\end{barticle}
\endbibitem

\bibitem[\protect\citeauthoryear{Liu et~al.}{2018}]{LiuSheLiCar2018High}
\begin{barticle}[author]
\bauthor{\bsnm{Liu},~\bfnm{Liu}\binits{L.}},
  \bauthor{\bsnm{Shen},~\bfnm{Yanyao}\binits{Y.}},
  \bauthor{\bsnm{Li},~\bfnm{Tianyang}\binits{T.}} \AND
  \bauthor{\bsnm{Caramanis},~\bfnm{Constantine}\binits{C.}}
(\byear{2018}).
\btitle{High dimensional robust sparse regression}.
\bjournal{arXiv preprint arXiv:1805.11643}.
\end{barticle}
\endbibitem

\bibitem[\protect\citeauthoryear{Nguyen and Tran}{2012}]{NguTra2012Robust}
\begin{barticle}[author]
\bauthor{\bsnm{Nguyen},~\bfnm{Nam~H}\binits{N.~H.}} \AND
  \bauthor{\bsnm{Tran},~\bfnm{Trac~D}\binits{T.~D.}}
(\byear{2012}).
\btitle{Robust lasso with missing and grossly corrupted observations}.
\bjournal{IEEE Transactions on Information Theory}
\bvolume{59}
\bpages{2036--2058}.
\end{barticle}
\endbibitem

\bibitem[\protect\citeauthoryear{Rivasplata}{2012}]{Riv2012Subgaussian}
\begin{barticle}[author]
\bauthor{\bsnm{Rivasplata},~\bfnm{Omar}\binits{O.}}
(\byear{2012}).
\btitle{Subgaussian random variables: An expository note}.
\end{barticle}
\endbibitem

\bibitem[\protect\citeauthoryear{She and Owen}{2011}]{SheOwe2011Outlier}
\begin{barticle}[author]
\bauthor{\bsnm{She},~\bfnm{Yiyuan}\binits{Y.}} \AND
  \bauthor{\bsnm{Owen},~\bfnm{Art~B}\binits{A.~B.}}
(\byear{2011}).
\btitle{Outlier detection using nonconvex penalized regression}.
\bjournal{Journal of the American Statistical Association}
\bvolume{106}
\bpages{626--639}.
\end{barticle}
\endbibitem

\bibitem[\protect\citeauthoryear{Tibshirani}{1996}]{Tib1996Regression}
\begin{barticle}[author]
\bauthor{\bsnm{Tibshirani},~\bfnm{Robert}\binits{R.}}
(\byear{1996}).
\btitle{Regression shrinkage and selection via the lasso}.
\bjournal{Journal of the Royal Statistical Society: Series B}
\bvolume{58}
\bpages{267--288}.
\end{barticle}
\endbibitem

\bibitem[\protect\citeauthoryear{Yuan and Lin}{2006}]{YuaLin2006Model}
\begin{barticle}[author]
\bauthor{\bsnm{Yuan},~\bfnm{Ming}\binits{M.}} \AND
  \bauthor{\bsnm{Lin},~\bfnm{Yi}\binits{Y.}}
(\byear{2006}).
\btitle{Model selection and estimation in regression with grouped variables}.
\bjournal{Journal of the Royal Statistical Society: Series B}
\bvolume{68}
\bpages{49--67}.
\end{barticle}
\endbibitem

\bibitem[\protect\citeauthoryear{Zou and
  Hastie}{2005}]{ZouHas2005Regularization}
\begin{barticle}[author]
\bauthor{\bsnm{Zou},~\bfnm{Hui}\binits{H.}} \AND
  \bauthor{\bsnm{Hastie},~\bfnm{Trevor}\binits{T.}}
(\byear{2005}).
\btitle{Regularization and variable selection via the elastic net}.
\bjournal{Journal of the Royal Statistical Society: Series B}
\bvolume{67}
\bpages{301--320}.
\end{barticle}
\endbibitem

\end{thebibliography}
\end{document}